\documentclass[12pt,letterpaper]{article}
\pdfoutput=1
\usepackage{graphics,epsfig,graphicx,color}
\graphicspath{{../Figures/}}
\usepackage[caption = false]{subfig}
\usepackage{float}
\usepackage{capt-of}
\usepackage{amssymb,latexsym}

\usepackage{setspace}

\usepackage{amsmath}

\usepackage{mathrsfs,amsfonts}

\usepackage{amsthm,amsxtra}

\usepackage[final]{showkeys}

\usepackage{stmaryrd}

\usepackage{algorithm}
\usepackage{algorithmicx}
\usepackage{algpseudocode}

\usepackage[openbib]{currvita}
\usepackage{bibentry}
\usepackage{etaremune}
\usepackage{enumitem} 

\newif\ifPDF
\ifx\pdfoutput\undefined
	\PDFfalse
\else
	\ifnum\pdfoutput > 0
		\PDFtrue
	\else
		\PDFfalse
	\fi
\fi

\ifPDF
	\usepackage{pdftricks}
	\begin{psinputs}
		\usepackage{pstricks}
		\usepackage{pstcol}
		\usepackage{pst-plot}
		\usepackage{pst-tree}
		\usepackage{pst-eps}
		\usepackage{multido}
		\usepackage{pst-node}
		\usepackage{pst-eps}
	\end{psinputs}
\else
	\usepackage{pstricks}
\fi

\ifPDF
	\usepackage[debug,pdftex,colorlinks=true, 
	linkcolor=blue, bookmarksopen=false,
	plainpages=false,pdfpagelabels]{hyperref}
\else
	\usepackage[dvips]{hyperref}
\fi

\usepackage{fancyhdr}

\usepackage{comment}

\usepackage[toc,page]{appendix}

\usepackage{cancel}

\usepackage{multirow}
\usepackage{tabularx}

\newtheorem{theorem}{Theorem}[section]
\newtheorem{lemma}[theorem]{Lemma}

\newtheorem{remark}[theorem]{Remark} 
\newtheorem{corollary}[theorem]{Corollary}

\newcommand{\dlim}{\displaystyle\lim}
\newcommand{\dint}{\displaystyle\int}

\newcommand{\eps}{\varepsilon}

\newcommand{\fc}{\mathfrak{c}}

\newcommand{\bbR}{\mathbb R} \newcommand{\bbS}{\mathbb S}

\newcommand{\bnu}{{\boldsymbol \nu}}

\newcommand{\bg}{\mathbf g}

 \newcommand{\bn}{\mathbf n}

 \newcommand{\bv}{\mathbf v} 
 \newcommand{\bx}{\mathbf x} 
\newcommand{\by}{\mathbf y} \newcommand{\bz}{\mathbf z}

 \newcommand{\bH}{\mathbf H}

\newcommand{\cA}{\mathcal A} 
\newcommand{\cC}{\mathcal C}  
 \newcommand{\cF}{\mathcal F}
 \newcommand{\cH}{\mathcal H}
 \newcommand{\cJ}{\mathcal J}
\newcommand{\cK}{\mathcal K} 
\newcommand{\cM}{\mathcal M}

\newcommand{\cU}{\mathcal U}

\newcommand{\aver}[1]{\langle {#1} \rangle}

\newcommand{\wt}{\widetilde}
\newcommand{\wh}{\widehat}

\newenvironment{keywords}
{\noindent{\bf Key words.}\small}{\par\vspace{1ex}}
\newenvironment{AMS}
{\noindent{\bf AMS subject classifications 2010.}\small}{\par}

\newcommand{\DELETE}[1]{}

\title{Unique determination of absorption coefficients in a semilinear transport equation}

\author{
	Kui Ren\thanks{
		Department of Applied Physics and Applied Mathematics, Columbia University, New York, NY 10027; kr2002@columbia.edu}
	\and
	Yimin Zhong\thanks{
		Department of Mathematics, University of California, Irvine, CA 92697; yiminz@uci.edu}
}

\begin{document}
%%%%%%%%%%%%%%%%%%%%%%%%%%%%%%%%%%%%%%%%%%%%%%%%%%%%%%%%%%%%%%%%%%
%%%%%%%BEGIN DOCUMENT %%%%BEGIN DOCUMENT %%%%BEGIN DOCUMENT%%%%%%%
%%%%%%%%%%%%%%%%%%%%%%%%%%%%%%%%%%%%%%%%%%%%%%%%%%%%%%%%%%%%%%%%%%

\maketitle

%\tableofcontents

%%%%%%%%%%%%%%%%%%%%%%%%%%%
%%%%%%%%ABSTRACT%%%%%%%%%%%
%%%%%%%%%%%%%%%%%%%%%%%%%%%

\begin{abstract}
	Motivated by applications in quantitative photoacoustic imaging, we study inverse problems to a semilinear radiative transport equation (RTE) where we intend to reconstruct absorption coefficients in the equation from single and multiple internal data sets. We derive uniqueness and stability results for the inverse transport problem in the absence of scattering (in which case we also derive some explicit reconstruction methods) and in the presence of known scattering.
\end{abstract}

%%%%%%%%%%%%%%%%%%%%%%%%%%%
%%%%%%%%%KEYWORDS%%%%%%%%%%
%%%%%%%%%%%%%%%%%%%%%%%%%%%


%%%%%%%%%%%%%%%%%%%%%%%%%%
%%%   AMS or PACS   %%%%%%
%%%%%%%%%%%%%%%%%%%%%%%%%%


%%%%%%%%%%%%%%%%%%%%%%%%%%%%%%%%%%%%%%%%%%%%%%%%%%%%%%%%%%%%%%%%%%
%%%%%% BEGINNING TEXT %%% BEGINNING TEXT %%% BEGINNING TEXT %%%%%%
%%%%%%%%%%%%%%%%%%%%%%%%%%%%%%%%%%%%%%%%%%%%%%%%%%%%%%%%%%%%%%%%%%

%%%%%%%%%%%%%%%%%%%%%%%%%%%%%%%%%%%%%%%%%%%%%%%%%%%%%%%%%%%%%%%%%%
%%%%%%%%%%%%%%%%%%%%%%%%%%%%%%%%%%%%%%%%%%%%%%%%%%%%%%%%%%%%%%%%%%
\section{Introduction}
\label{SEC:Intro}
%%%%%%%%%%%%%%%%%%%%%%%%%%%%%%%%%%%%%%%%%%%%%%%%%%%%%%%%%%%%%%%%%%
%%%%%%%%%%%%%%%%%%%%%%%%%%%%%%%%%%%%%%%%%%%%%%%%%%%%%%%%%%%%%%%%%%

Let $\Omega\subseteq\bbR^d$ ($d\ge 2$) be a domain with boundary $\partial\Omega$, and $\bbS^{d-1}$ the unit sphere in $\bbR^d$. We define the phase space $X:=\Omega\times\bbS^{d-1}$ and the incoming boundary of the phase space $\Gamma_-:=\{(\bx, \bv)\ |\  (\bx, \bv)\in\partial\Omega\times\bbS^{d-1}\ \mbox{s.t.}\ -\bnu(\bx)\cdot \bv > 0\}$, $\bnu(\bx)$ being the unit outer normal vector at $\bx\in\partial\Omega$. We are interested in the semilinear radiative transport equation:
\begin{equation}\label{EQ:ERT TP}
\begin{array}{rcll}
\bv \cdot \nabla u + (\sigma_a+\sigma_s) u(\bx, \bv) + \sigma_b \aver{u} u(\bx, \bv) &=& \sigma_s(\bx) K u(\bx, \bv), &\mbox{in}\ X\,\\
u(\bx, \bv) &= & g(\bx, \bv), &\mbox{on}\ \Gamma_{-}\,
\end{array}
\end{equation}
where $\aver{u}$ denotes the average of $u(\bx, \bv)$ over the variable $\bv$, that is, 
\begin{equation}\label{EQ:Aver}
	\aver{u}:=\int_{\bbS^{d-1}} u(\bx, \bv) d\bv\,,
\end{equation}
with $d\bv$ being the \emph{normalized} surface measure on $\bbS^{d-1}$. The linear operator $K$ is defined through the relation
\begin{equation}\label{EQ:Scattering Op}
\begin{aligned}
K u(\bx, \bv) := \int_{\bbS^{d-1}} \Theta(\bv, \bv') u(\bx, \bv') d\bv',
\end{aligned}
\end{equation}
with the kernel $\Theta(\bv, \bv')$ being symmetric and satisfying the normalization conditions 
\[
\dint_{\bbS^{d-1}} \Theta(\bv, \bv') d\bv' = \dint_{\bbS^{d-1}} \Theta(\bv, \bv') d\bv = 1.
\]

Transport equations such as ~\eqref{EQ:ERT TP} often appear in the literature as the mathematical models to describe radiative transfer processes in heterogeneous media. We are interested in the application of this equation in modeling the propagation of near infra-red photons in biological tissues~\cite{Arridge-IP99,Bal-IP09,Ren-CiCP10}. In such a case, $u(\bx, \bv)$ denotes the density of the photons at position $\bx$ traveling in direction $\bv$. The coefficients $\sigma_a(\bx)$ and $\sigma_s(\bx)$ are the usual single-photon absorption and scattering coefficients respectively, and the kernel $\Theta(\bv, \bv')$ describes the probability of photons traveling in direction $\bv'$ getting scattered into direction $\bv$. The coefficient $\sigma_b(\bx)$ is called the two-photon absorption coefficient. It is used to model the two-photon absorption process, that is, the phenomenon that an electron transfers to an excited state after simultaneously absorbing two photons whose total energy exceed the electronic energy band gap. Such two-photon absorption process can also be viewed as a regular physical absorption process whose effective absorption strength, $\sigma_b \aver{u}$, depends on the local density of the photons. We refer interested readers to~\cite{BaReZh-JBO18,ReZh-SIAM18} and references therein for more details on the modeling of two-photon absorption in diffusive media.

In the rest of this paper, we study an inverse problem to the transport model~\eqref{EQ:ERT TP} where we intend to reconstruct the absorption coefficients $\sigma_a$ and $\sigma_b$ from internal data of the form
\begin{equation}\label{EQ:Data TP}
	H(\bx) := \sigma_a(\bx)\aver{u}(\bx) +\sigma_b(\bx) \aver{u}^2(\bx),\ \ \bx\in\bar\Omega\,.
\end{equation}
Such inverse problems originate from applications in quantitative photoacoustic imaging where internal data~\eqref{EQ:Data TP} can be obtained from photoacoustic measurements; see~\cite{BaJoJu-IP10,MaRe-CMS14} and references therein for recent developments in the field. In the diffusive regime, that is, when the transport model~\eqref{EQ:ERT TP} is replaced with its diffusion approximation, it has been shown in~\cite{BaReZh-JBO18,ReZh-SIAM18} that one can reconstruct all three coefficients $(\sigma_a, \sigma_b, \sigma_s)$ from a finite set of internal data of the form~\eqref{EQ:Data TP}. The objective of this work is to show that one can reconstruct uniquely $(\sigma_a, \sigma_b)$ in the semilinear transport equation~\eqref{EQ:ERT TP} from two sets of internal data.

The inverse problem we described above is closely related to an inverse problem to the linear transport equation:
\begin{equation}\label{EQ:ERT}
\begin{array}{rcll}
\bv \cdot \nabla u + (\Sigma_a+\sigma_s) u(\bx, \bv) &=& \sigma_s(\bx) K u(\bx, \bv), &\mbox{in}\ X\,\\
u(\bx, \bv) &= & g(\bx, \bv), &\mbox{on}\ \Gamma_{-}\,
\end{array}
\end{equation}
with data of the form
\begin{equation}\label{EQ:Data}
	H(\bx) := \Sigma_a(\bx)\aver{u}(\bx),\ \ \bx\in\bar\Omega\,.
\end{equation}
In fact, the semilinear transport equation~\eqref{EQ:ERT TP} can be viewed as the linear transport equation~\eqref{EQ:ERT} whose absorption coefficient $\Sigma_a$ depends on the density $\aver{u}$ in a linear manner: $\Sigma_a=\sigma_a+\sigma_b \aver{u}$.

Inverse problems to the radiative transport equation have been studied extensively in the past two decades; see for instance~\cite{Bal-IP09} for a recent review on the topic. Most of existing analytical and computational results are on the linear transport equation~\eqref{EQ:ERT}. These include, but not limited to, problems where boundary data encoded in the map $u_{|\Gamma_-}\mapsto u_{|\Gamma_+}$ and alike are available~\cite{BaJo-IPI08,BaMo-SIAM12,ChLiWa-IP18,ChSt-IP96,ChSt-CPDE96,Chung-arXiv20,DiRe-JCP14,GoYa-SIAM16,KiMo-IP06,LaLiUh-SIAM18,ReBaHi-SIAM06,StUh-MAA03,StUh-APDE08,Tamasan-IP02,Wang-AIHP99,ZhZh-SIAM18}, as well as problems where internal data of the type~\eqref{EQ:Data}~\cite{BaJoJu-IP10,MaRe-CMS14,ReZhZh-IP15,WaZh-IP18,HaNeNgRa-SIAM18,TaCoKaAr-IP12,SaTaCoAr-IP13} and alike~\cite{ChSc-SIAM17,LiYaZh-IP20} are available. Existing results, either analytical or computational, on nonlinear transport models such as~\eqref{EQ:ERT TP} are very limited; see~\cite{ChGaRe-JCP11,LiSu-arXiv19,KlLaLi-arXiv20} for some related results.

%%%%%%%%%%%%%%%%%%%%%%%%%%%%%%%%%%%%%%%%%%%%%%%%%%%%%%%%%%%%%%%%%%
%%%%%%%%%%%%%%%%%%%%%%%%%%%%%%%%%%%%%%%%%%%%%%%%%%%%%%%%%%%%%%%%%%
\section{The forward problem}
\label{SEC:Forward}
%%%%%%%%%%%%%%%%%%%%%%%%%%%%%%%%%%%%%%%%%%%%%%%%%%%%%%%%%%%%%%%%%%
%%%%%%%%%%%%%%%%%%%%%%%%%%%%%%%%%%%%%%%%%%%%%%%%%%%%%%%%%%%%%%%%%%

We start by establishing the well-posedness theory for the semilinear transport equation~\eqref{EQ:ERT TP}. 
To setup the analysis, we denote by $L^p(X)$ (resp. $L^p(\Omega)$) the space of real-valued functions whose $p$-th power are Lebesgue integrable on $X$ (resp. $\Omega$), and $\cH^p(X)$ the space of $L^p(X)$ functions whose derivative in direction $\bv$ is in $L^p(X)$, i.e. $\cH^p(X)=\{f(\bx,\bv): f\in L^p(X)\ \mbox{and}\ \bv\cdot\nabla f\in L^p(X)\}$. We denote by $L^p(\Gamma_-)$ the space of functions that are traces of $\cH^p(X)$ functions on $\Gamma_-$ under the norm $\|f\|_{L^p(\Gamma_-)}=(\int_{\partial\Omega}\int_{\bbS_{\bx-}^{d-1}}|\bn(\bx)\cdot\bv||f|^p d\bv d\gamma)^{1/p}$, $d\gamma$ being the surface measure on $\partial\Omega$ and $\bbS_{\bx-}^{d-1}=\{\bv: \bv\in\bbS^{d-1}\ \mbox{s.t.}\ -\bnu(\bx)\cdot\bv>0\}$. It is well-known~\cite{Agoshkov-Book98, DaLi-Book93-6} that both $\cH^p(X)$ and $L^p(\Gamma_-)$ are well-defined. 

For a given set $Y$, we introduce the space of bounded functions on $Y$:
\[
		\cF_{\underline f}^{\overline f}(Y):=\{ f\in L^\infty(Y) \mid \exists \underline f, \overline f \ \mbox{s.t.}\ 0<\underline f \le f\le \overline f<+\infty\ a.e.\}\, .
\]

\emph{Unless stated otherwise, we make the following assumptions on the domain $\Omega$, the scattering coefficient and the scattering phase function throughout the paper}:\\[1ex]
$(\cA)$. (i) the domain $\Omega$ is bounded, convex and smooth; (ii) the scattering coefficient $\sigma_s(\bx)\in \cF_{\underline \sigma_s}^{\overline \sigma_s}(\Omega)$ for some constants $\underline\sigma_s$ and $\overline\sigma_s$; and (iii) the scattering phase function $\Theta(\bv, \bv')\in \cF_{\underline \theta}^{\overline\theta}(\bbS^{d-1}\times \bbS^{d-1})$ for some $\underline\theta$ and $\overline\theta$.\\[1ex]
With the convention that the surface measure $d\bv$ on $\bbS^{d-1}$ is normalized, we observe that this assumption means that $\underline\theta\le 1$ while $\overline\theta\ge 1$.

%we make the following assumptions on the coefficients and scattering phase function:\\[1ex]%\vskip 0mm
%$(\cA)$. \hskip 3cm .\\[1ex]

%(ii) There exists positive constants $\underline{\sigma}_a$, $\overline{\sigma}_a$, $\underline{\sigma}_b$, $\overline{\sigma}_b$, $\overline{\sigma}_s$, $\underline{\theta}$, $\overline{\theta}$, $\underline\Sigma_a$ and $\overline\Sigma_a$ such that $\sigma_a\in \cF_{\underline \sigma_a}^{\overline \sigma_a}$, $\sigma_b\in \cF_{\underline \sigma_b}^{\overline \sigma_b}$, $\sigma_s\in \cF_{0}^{\overline \sigma_s}$, $\theta\in \cF_{\underline \theta}^{\overline \theta}$, and $\Sigma_a\in \cF_{\underline \Sigma_a}^{\overline \Sigma_a}$ where

%(iii) $\sigma_a|_{\partial\Omega}$, $\sigma_b|_{\partial\Omega}$ and $\Sigma_a|_{\partial\Omega}$ are known.\\[1ex]

For any point $(\bx, \bv)\in X$, we use $\tau_{-}(\bx, \bv)$ to denote the distance a particle starting from $\bx$ and traveling in the direction $-\bv$ has to travel to reach the boundary of the domain. That is:
\begin{equation}\label{EQ:Tau}
	\tau_{-}(\bx, \bv) := \sup\{s\in\bbR \mid \bx - s\bv\in\Omega\}\,.
\end{equation}
Note that due to the assumption that $\Omega$ is convex, $\tau_{-}(\bx, \bv)$ is uniquely determined for any $(\bx, \bv)\in X$.

The following simple result on the linear transport equation~\eqref{EQ:ERT} turns out to be useful.
\begin{lemma}\label{LEM:Positivity Lin ERT}
	 Let $g\in L^{\infty}(\Gamma_{-})$ be given such that $\underline g:=\inf_{\Gamma_-}g>0$ and $u$ be the unique solution to~\eqref{EQ:ERT} with $(\Sigma_a, \sigma_s, \Theta)$. Assume that $\Sigma_a\in\cF_{\underline\Sigma_a}^{\overline\Sigma_a}(\Omega)$ for some $\underline\Sigma_a$ and $\overline\Sigma_a$. Then, under the assumptions in $(\cA)$, there exists some constant $\fc >0$ such that $u \ge \fc>0$.
\end{lemma}
\begin{proof}
	We first observe that with all the assumptions made, we have that $u\ge 0$ from the standard transport theory~\cite{DaLi-Book93-6}. Let $\wt u$ be the solution to
\begin{equation*}
\begin{array}{rcll}
\bv \cdot \nabla \wt u(\bx, \bv) + (\Sigma_a + \sigma_s) \wt u(\bx, \bv) &=& 0, &\mbox{in}\ X\, \\
\wt u(\bx, \bv) &= & g(\bx, \bv), &\mbox{on}\ \Gamma_{-}\,.
\end{array}
\end{equation*}
Then $\wt u$ can be found analytically as
\[
	\wt u(\bx, \bv)=g(\bx-\tau_{-}(\bx, \bv)\bv, \bv)\exp(-\int_0^{\tau_{-}(\bx, \bv)}  (\Sigma_a + \sigma_s) (\bx-t \bv) dt),
\]
where $\tau_{-}(\bx, \bv)$ has been defined in~\eqref{EQ:Tau}. This expression implies, together with the facts that $\Omega$ is bounded and $g\ge \underline{g}$, that $\wt u(\bx, \bv) \ge \fc'$ for some $\fc'>0$. 

We then check that $\phi:=u-\wt u$ solves the linear transport equation
\begin{equation*}
\begin{array}{rcll}
\bv \cdot \nabla \phi + (\Sigma_a+\sigma_s) \phi  &=& \sigma_s(\bx) K u, &\mbox{in}\ X\, \\
\phi(\bx, \bv) &= & 0, &\mbox{on}\ \Gamma_{-}\,.
\end{array}
\end{equation*}
Using the fact that $u\ge 0$ (and therefore $\sigma_s K u\ge 0$), we conclude that this transport equation has a solution $\phi\ge 0$. Therefore, $u\ge \wt u$. The final result then follows.
\end{proof}

We are now ready study solutoin properties of the semilinear transport equation~\eqref{EQ:ERT TP}. For physical reasons, we are only interested in \emph{non-negative} solutions. We consider two types of incoming boundary sources.

%%%%%%%%%%%%%%%%%%%%%%%%%%%%%%%%%%%%%%%%%%%%%%%%%%%%%%%%%%%%%%%%%%
\subsection{General bounded sources}
%%%%%%%%%%%%%%%%%%%%%%%%%%%%%%%%%%%%%%%%%%%%%%%%%%%%%%%%%%%%%%%%%%

For a given set of functions $(\sigma_a, \sigma_b, \sigma_s, \Theta)$, let $g(\bx, \bv)\in L^{\infty}(\Gamma_{-})$ be the boundary source for~\eqref{EQ:ERT TP}. We denote by $\underline{g}:=\inf_{(\bx, \bv)\in\Gamma_-} g(\bx, \bv)$ and $\overline{g}:=\sup_{(\bx, \bv)\in\Gamma_-} g(\bx, \bv)$. We assume that
\begin{equation}\label{EQ:SOURCE}
\underline{g}>0,\quad\mbox{and},\quad 
\overline g\le \left\{
	\begin{array}{cl}
		\displaystyle\inf_\Omega \dfrac{\sigma_a}{\sigma_b},& \mbox{when}\ \sigma_s\equiv 0,\\
		\max\Big(\displaystyle\inf_\Omega \dfrac{\sigma_a}{\sigma_b},\ 2\underline{\theta}  \displaystyle\inf_{\Omega}\frac{\sigma_s(\bx)}{\sigma_b(\bx)}\Big),& \mbox{when}\ \sigma_s\neq 0.
	\end{array}\right.
\end{equation}
with $\underline{\theta}$ being the constant introduced in Assumption $(\cA)$. We can show that a non-negative solution to~\eqref{EQ:ERT TP} with such a $g$ exists and is unique. Our main strategy of proof is to analyze the fixed-point iteration: $k\ge 1$
\begin{equation}\label{EQ:ERT TP FPI}
\begin{array}{rcll}
\bv \cdot \nabla u^{k} + (\sigma_a+\sigma_s) u^{k}(\bx, \bv) + \sigma_b \aver{u^{k-1}} u^{k}(\bx, \bv) &=& \sigma_s(\bx) K u^{k}(\bx, \bv), &\mbox{in}\ X\,\\
u^k(\bx, \bv) &= & g(\bx, \bv), &\mbox{on}\ \Gamma_{-}\,
\end{array}
\end{equation}
using Kellogg's uniqueness theory~\cite{Kellogg-PAMS76} for the Schauder Fixed-Point Theorem together with the averaging lemma~\cite{GoLiPeSe-JFA88}. For the convenience of the readers, we recalled both results in the Appendix~\ref{SEC:Appendix A}.

We now prove the existence and uniqueness of non-negative solutions to~\eqref{EQ:ERT TP}.
\begin{theorem}\label{THM:WELL}
	For any $\sigma_a\in\cF_{\underline\sigma_a}^{\overline\sigma_a}$ and $\sigma_b\in\cF_{\underline\sigma_b}^{\overline\sigma_b}$, let $g\in L^{\infty}(\Gamma_{-})$ be given as in~\eqref{EQ:SOURCE}. Then, under the assumption $(\cA)$, the transport equation~\eqref{EQ:ERT TP} has a unique bounded solution $u\in L^\infty(X)$ that is non-negative: $u(\bx, \bv)\ge 0$. 
\end{theorem}
\begin{proof}
    We first show the existence of non-negative solution by showing that a fixed point exist for the iteration~\eqref{EQ:ERT TP FPI}. We introduce a set of bounded functions:
\begin{equation*}
\cM = \{m\in L^2(\Omega) \mid 0\le m(\bx) \le \overline{g}\ a.e.\}\,.
\end{equation*}
It is clear that $\cM$ is convex, bounded and closed under the $L^2$ topology.
 
For any given function $m(\bx)\in \cM$, we introduce a linear transport equation:
\begin{equation}\label{EQ:MODIF ERT}
\begin{array}{rcll}
\bv \cdot \nabla u + (\sigma_a+\sigma_s) u(\bx, \bv) + \sigma_b(\bx) m(\bx) u(\bx, \bv) &=& \sigma_s(\bx) K u(\bx, \bv), &\mbox{in}\ X\, \\
u(\bx, \bv) &= & g(\bx, \bv), &\mbox{on}\ \Gamma_{-}\,.
\end{array}
\end{equation}
This is simply the semilinear transport equation~\eqref{EQ:ERT TP} with $\aver{u}$ replaced by $m(\bx)$. Under the assumptions we have made, we conclude from the standard transport theory~\cite{Agoshkov-Book98, DaLi-Book93-6} that this linear transport equation has a unique solution $u(\bx, \bv)\in L^{\infty}(X)$. Moreover, $u$ satisfies $0\le u(\bx, \bv)\le \overline{g}$. This means also that $0\le \aver{u}\le \overline{g}$.

Therefore, the operator $\cC: m\mapsto \aver{u}$, defined through the relation
\begin{equation}\label{EQ:Operator C}
	\cC(m): = \aver{u}
\end{equation}
with $u$ being the solution to~\eqref{EQ:MODIF ERT}, maps $\cM$ into a subset of it, that is, $\cC(\cM)\subseteq \cM$. Meanwhile, we can also verify that $\cC$ is a continuous operator on $\cM$. To see that, let $u$ and $\tilde u$ be the solutions of~\eqref{EQ:MODIF ERT} with $m$ and $\tilde m$ respectively. Then $w=\tilde u- u$ solves
\begin{equation*}
\begin{array}{rcll}
\bv \cdot \nabla w + (\sigma_a+\sigma_s) w + \sigma_b(\bx) m(\bx) w &=& \sigma_s(\bx) K w -(m-\tilde m)\tilde u, &\mbox{in}\ X\,\\
w(\bx, \bv) &= & 0, &\mbox{on}\ \Gamma_{-}\,.
\end{array}
\end{equation*}
With the assumptions we have, especially the fact that $m\ge 0$, this transport equation admits a unique solution that satisfies the stability bound 
\[
	\|w\|_{\cH^1(X)}:=\|u-\tilde u\|_{\cH^1(X)} \le \wt\fc\|(m-\tilde m)\tilde u\|_{L^2(X)}\le \fc\|(m-\tilde m)\|_{L^2(\Omega)}
\]
for some constants $\fc, \tilde \fc>0$. The last inequality comes from the fact that $\tilde u \in L^\infty(X)$. Using this bound, together with the averaging lemma~\cite{GoLiPeSe-JFA88}, that is, Theorem~\ref{THM:Averaging Lemma}, and the fact that $w_{|\Gamma_-}:=u_{|\Gamma_-}-\tilde u_{|\Gamma_-}=0$, we conclude that $\aver{w}:=\aver{u}-\aver{\tilde u} \in W^{1/2,2}(\Omega)$ and
\[
	\|\cC(m)-\cC(\tilde m)\|_{W^{1/2,2}(\Omega)}:=\|\aver{u}-\aver{\tilde u}\|_{W^{1/2,2}(\Omega)}\le \|u-\tilde u\|_{\cH^1(X)} \le \fc\|(m-\tilde m)\|_{L^2(\Omega)}.
\]
This bound, combined with the Kondrachov embedding theorem~\cite{AdFo-Book03}, leads to the fact that the operator $\cM$ is a continuous compact operator from $\cM$ to itself. The Schauder Fixed-Point Theorem~\cite{GiTr-Book00,Zeidler-Book95} then implies that exists a fixed point $m^{\ast}\in \cM$ that $\cC (m^{\ast}) = m^{\ast}$. Therefore, there exists a bounded non-negative solution to the transport equation~\eqref{EQ:ERT TP}. 

We now use Kellogg's theory~\cite{Kellogg-PAMS76}, that is, Theorem~\ref{THM:UNIQ}, to show uniqueness of the above fixed point. To verify that the fixed point cannot live on $\partial\cM$, we observe that since $\sigma_a>0$, the solution operator of~\eqref{EQ:MODIF ERT} is a strict contraction even when $m\equiv 0$. Therefore $\aver{u}<\overline{g}$. Meanwhile, Lemma~\ref{LEM:Positivity Lin ERT} implies that $\aver{u}>0$. Therefore, $\cC$ maps $\cM$ into its interior $\cM^\circ$. This shows that the fixed point of $\cC$ cannot live on $\partial\cM$.

The remaining task is to show that the Frech\'{e}t derivative of $\cC$ does not have $1$ as its eigenvalue in $\cM$. Let $u$ be the solution to~\eqref{EQ:MODIF ERT} with function $m(\bx)$, $\delta m(\bx)$ a perturbation of $m$ such that $m+\delta m\in\cM$, and $\phi$ the solution to
\begin{equation}\label{EQ: DERIV}
\begin{array}{rcll}
\bv\cdot \nabla \phi + (\sigma_a+\sigma_s) \phi(\bx, \bv) + \sigma_b m \phi(\bx,\bv) &=& \sigma_s K\phi - \sigma_b \delta m(\bx) u(\bx, \bv),
 &\mbox{in}\ X\, \\
\phi(\bx, \bv) &= & 0, &\mbox{on}\ \Gamma_{-}\,.
\end{array}
\end{equation}
Then it is straightforward to verify that the Frech\'{e}t derivative of $\cC$ at $m$ in the direction $\delta m$ is given as $\cC'[m](\delta m) = \aver{\phi}$. Assume now that $\cC'[m]$ indeed has $1$  as its eigenvalue and let $\aver{\phi}\not\equiv 0$ be the corresponding eigenfunction, i.e., $\cC'[m](\aver{\phi}) = \aver{\phi}$. Then the transport equation~\eqref{EQ: DERIV} is equivalent to
\begin{equation*}
\begin{array}{rcll}
\bv\cdot \nabla \phi + (\sigma_a+\sigma_s) \phi(\bx, \bv) + \sigma_b m \phi (\bx, \bv)&=& \sigma_s K\phi - \sigma_b \aver{\phi} u(\bx, \bv),
&\mbox{in}\ X\,, \\
\phi(\bx, \bv) &= & 0, &\mbox{on}\ \ \Gamma_{-}\,.
\end{array}
\end{equation*}
Let $\sigma_t(\bx):= \sigma_a(\bx) + \sigma_s(\bx)+ \sigma_b(\bx) m(\bx) $ and $R(\bx, \bv):= \sigma_s K \phi - \sigma_b\aver{\phi}u(\bx, \bv)$. By the standard method of characteristics, it is straightforward to check that $\phi$ satisfies
\begin{equation}\label{EQ:Contraction}
\begin{aligned}
|\phi(\bx, \bv)| &= \left|\int_{0}^{\tau_{-}(\bx, \bv)}  \exp\left[-\int_{0}^\ell \sigma_t(\bx - s \bv) ds \right] R(\bx - \ell\bv, \bv) d\ell \right|\\
&= \left| \int_{0}^{\tau_{-}(\bx, \bv)} \sigma_t(\bx-\ell\bv) \exp\left[-\int_{0}^\ell \sigma_t(\bx - s \bv) ds \right] \frac{R(\bx - \ell\bv, \bv)}{\sigma_t(\bx - \ell\bv)} d\ell\right| \\
&\le \int_{0}^{\tau_{-}(\bx, \bv)} \sigma_t(\bx-\ell\bv) \exp\left[-\int_{0}^\ell \sigma_t(\bx - s \bv) ds \right] \sup_{\ell\in(0,\tau_{-}(\bx,\bv))}\left|\frac{R(\bx - \ell\bv, \bv)}{\sigma_t(\bx - \ell\bv)}\right| d\ell\\
&\le \left(1 - \exp\left[-\int_{0}^{\tau_{-}(\bx, \bv)}\sigma_t(\bx - s\bv) ds \right]\right) \sup_{(\by, \bv)\in X}\frac{\left|R(\by, \bv)\right|}{\sigma_t(\by)}\\
& \le \beta \sup_{(\by, \bv)\in X}\frac{\left|R(\by, \bv)\right|}{\sigma_t(\by)},\qquad \mbox{for some}\ \beta<1 \,.
\end{aligned}
\end{equation}

When $\sigma_s\equiv 0$, we have
\begin{equation*}
|R(\bx, \bv)| = \sigma_b u(\bx,\bv)|\aver{\phi}|\le \sigma_b\, u(\bx,\bv)\, \sup_X |\phi|.
\end{equation*}
This, together with~\eqref{EQ:Contraction}, gives that
\begin{equation}\label{EQ:Temp1}
	|\phi(\bx, \bv)|\le \beta \sup_{X}\frac{\sigma_b u(\bx, \bv)}{\sigma_t}\, \sup_X |\phi|\le  \beta\overline g\,  \sup_{\Omega}\frac{\sigma_b}{\sigma_t}\, \sup_X |\phi|.
\end{equation}
When $\overline{g}\le \inf_\Omega\dfrac{\sigma_a}{\sigma_b}$, we have that $|\phi(\bx, \bv)| \le \beta \sup_X |\phi|$, $\forall\, (\bx, \bv)\in X$. Therefore $\phi \equiv 0$.

When $\sigma_s\not\equiv 0$ satisfies the assumption $(\cA)$, we have
\begin{equation*}
|R(\bx, \bv)| \le |\sigma_s K \phi|+|\sigma_b u(\bx, \bv) \aver{\phi}| \le \Big(\sigma_s +\sigma_b u(\bx, \bv)\Big)\sup_X |\phi|.
\end{equation*}
This, together with~\eqref{EQ:Contraction}, gives that
\begin{equation}\label{EQ:Temp2}
	|\phi(\bx, \bv)|\le \sup_{X}\frac{\sigma_s +\sigma_b u(\bx, \bv)}{\sigma_t}\, \sup_X |\phi| \le \sup_{\Omega}\frac{\sigma_s +\sigma_b \overline g}{\sigma_t}\, \sup_X|\phi|.
\end{equation}
Therefore, when $\overline{g}\le \inf_\Omega\frac{\sigma_a}{\sigma_b}$, we have that $|\phi(\bx, \bv)| \le \beta \sup_X |\phi|$, $\forall\, (\bx, \bv)\in X$, for some $\beta<1$. Therefore $\phi \equiv 0$.

Meanwhile, we can also have
\begin{multline*}
|R(\by, \bv)| =|\sigma_s K\phi-\underline\theta \sigma_s \aver{\phi} + \underline\theta \sigma_s \aver{\phi} -\sigma_b u \aver{\phi}|\le 
|\sigma_s K \phi-\underline\theta \sigma_s \aver{\phi}| + |\underline\theta \sigma_s \aver{\phi} -\sigma_b u \aver{\phi}|\\
\le \left((1-\underline{\theta})\sigma_s + |\sigma_s(\by) \underline{\theta} - \sigma_b(\by) u(\by, \bv)|\right)  \sup |\phi|.
\end{multline*}
This, together with~\eqref{EQ:Contraction}, gives that
\begin{equation}\label{EQ:Temp3}
	|\phi(\bx, \bv)|\le \sup_{(\by, \bv)\in X}\frac{\left( (1-\underline{\theta})\sigma_s +|\sigma_s(\by) \underline{\theta} - \sigma_b(\by) u(\by, \bv)| \right)}{\sigma_t}\, \sup|\phi|.
\end{equation}
Therefore, when $\overline g\le 2\underline{\theta}  \displaystyle\inf_{\Omega}\frac{\sigma_s(\bx)}{\sigma_b(\bx)}$, we have that $|\phi(\bx, \bv)| \le \sup_\Omega\frac{\sigma_s}{\sigma_t} \sup_X |\phi|$, $\forall\, (\bx, \bv)\in X$. Therefore $\phi \equiv 0$.

We have thus shown that $1$ is not an eigenvalue of $\cC'[m]$ in $\cM$. Therefore, the fixed-point of $\cC$ in $\cM$ is unique. This concludes the proof.
\end{proof}

\begin{remark}\rm
The above theory requires the smallness of the boundary source $g(\bx, \bv)$ (as a sufficient condition) for the solution to the transport equation~\eqref{EQ:ERT TP} to be unique. This type of smallness assumptions is common for nonlinear problems. Note that in the diffusive limit when $\sigma_s\to +\infty$, the ratio $\sigma_s/\sigma_b \to +\infty$. This means that the smallness requirement is not necessary anymore in the diffusive regime. This is exactly what happened in~\cite{ReZh-SIAM18} where it is shown that the diffusion approximation of~\eqref{EQ:ERT TP} has a unique non-negative solution for any given bounded non-negative boundary source.%From the practical point of view, the two-photon absorption coefficient $\sigma_b$ in water on the order of $1 \mbox{cm/GW}$ and $\underline{\theta}\sigma_s$ is on the order of $10 \textrm{cm}^{-1}$~\cite{NiAn-CPL81}. This leads to an upper bound on the source intensity at about $20\mbox{GW}/\mbox{cm}^2$. This is relatively high intensity in practice.
\end{remark}

The following fact about non-negative solutions to the transport equation~\eqref{EQ:ERT TP} can be proved using the same ideas of Lemma~\ref{LEM:Positivity Lin ERT}.
\begin{corollary}\label{COR:Positivity}
		For any $\sigma_a\in\cF_{\underline\sigma_a}^{\overline\sigma_a}$ and $\sigma_b\in\cF_{\underline\sigma_b}^{\overline\sigma_b}$, let $g\in L^{\infty}(\Gamma_{-})$ be given as in~\eqref{EQ:SOURCE} and $u$ be the corresponding unique non-negative solution to~\eqref{EQ:ERT TP}. Then, under the assumption $(\cA)$, $u\ge \fc$ for some $\fc>0$.
\end{corollary}
\begin{proof}
	This result can be seen from two comparisons between solutions. Let $w$ be the solution to the linear transport equation
\begin{equation}\label{EQ:W2}
\begin{array}{rcll}
\bv \cdot \nabla w(\bx, \bv) + (\sigma_a+\sigma_s) w(\bx, \bv) + \sigma_b \overline{g} w(\bx, \bv) &=& \sigma_s K u, &\mbox{in}\ \ X\, \\
w(\bx, \bv) &= & g(\bx, \bv), &\mbox{on}\ \ \Gamma_{-}\,.
\end{array}
\end{equation}
Using the fact that $u\ge 0$, we conclude that $\sigma_s K u\ge 0$, and therefore $w\ge 0$. Let $\phi :=u-w$. Then $\phi$ solves
\begin{equation}\label{EQ:Comparison1}
\begin{array}{rcll}
\bv \cdot \nabla \phi + (\sigma_a+\sigma_s)\phi + \sigma_b \aver{u} \phi &=& \sigma_b (\overline{g}-\aver{u}) w, &\mbox{in}\ \ X\, \\
\phi(\bx, \bv) &= & 0, &\mbox{on}\ \ \Gamma_{-}\,.
\end{array}
\end{equation}
The right-hand-side of the equation is clearly non-negative (since $\overline{g}\ge \aver{u}$ and $w\ge 0$). Therefore $\phi\ge 0$. This implies that $u\ge w$.

Next, let $\wt w$ be the solution to~\eqref{EQ:W2} with the right-hand-side removed, that is, $\wt w$ solves
\[
\begin{array}{rcll}
\bv \cdot \nabla \wt w(\bx, \bv) + (\sigma_a+\sigma_s) \wt w(\bx, \bv) + \sigma_b \overline{g} \wt w(\bx, \bv) &=& 0, &\mbox{in}\ \ X\, \\
w(\bx, \bv) &= & g(\bx, \bv), &\mbox{on}\ \ \Gamma_{-}\,
\end{array}
\]
Then, $\wt w$ can be written as
\[
	\wt w(\bx, \bv)=g(\bx-\tau_{-}(\bx, \bv)\bv, \bv)\exp\Big(-\int_0^{\tau_{-}(\bx, \bv)}  (\sigma_a + \sigma_s + \sigma_b \overline g) (\bx-t \bv) dt\Big).
\]

We therefore have that $\wt w\ge \fc:=\underline g e^{-(\overline\sigma_a+\overline\sigma_s+\overline{g}\overline{\sigma_b}){\rm diam}(\Omega)}>0$.
Let $\wt \phi := w-\wt w$, then $\wt \phi$ solves
\begin{equation*}
\begin{array}{rcll}
\bv \cdot \nabla \wt \phi + (\sigma_a+\sigma_s) \wt \phi + \sigma_b \overline{g} \wt\phi &=& \sigma_s K u, &\mbox{in}\ \ X\, \\
\wt \phi(\bx, \bv) &= & 0, &\mbox{on}\ \ \Gamma_{-}\,.
\end{array}
\end{equation*}
Non-negativity of $\sigma_a K u$ then implies that $\wt \phi\ge 0$. This gives that $w\ge \wt w$. We are now able to conclude that $u\ge w\ge \wt w\ge \fc>0$.
\end{proof}

%%%%%%%%%%%%%%%%%%%%%%%%%%%%%%%%%%%%%%%%%%%%%%%%%%%%%%%%%%%%%%%%%%
\subsection{Collimated sources}
%%%%%%%%%%%%%%%%%%%%%%%%%%%%%%%%%%%%%%%%%%%%%%%%%%%%%%%%%%%%%%%%%%

We now consider the transport equation~\eqref{EQ:ERT TP} with collimated illumination sources of the form:
\begin{equation}\label{EQ:Collimated Source}
g(\bx, \bv) = \mathfrak{g}(\bx) \delta(\bv - \bv'),\quad \bv'\in\bbS^{d-1}_{\bx-},
\end{equation}
where $\mathfrak g(\bx)\ge 0$ on $\partial\Omega$. This is a type of illumination strategies that is practically important.

By analyzing again the fixed-point iteration~\eqref{EQ:ERT TP FPI}, we can establish the following existence and uniqueness of non-negative solutions to~\eqref{EQ:ERT TP} with this new boundary source.
\begin{theorem}
For any $\sigma_a\in\cF_{\underline\sigma_a}^{\overline\sigma_a}$, $\sigma_b\in\cF_{\underline\sigma_b}^{\overline\sigma_b}$, and $(\sigma_s, \Theta)$ satisfying the assumptions in $(\cA)$, let $\mu := \displaystyle\sup_{\bx\in\Omega}\frac{\sigma_s}{\sigma_{a}+\sigma_s}$ and $\kappa := \displaystyle\sup_{\bx\in\Omega}\frac{\sigma_b}{\sigma_a + \sigma_s}$. Assume that $\overline{\mathfrak{g}}:= \displaystyle\sup_{\bx\in\partial\Omega} \mathfrak{g}(\bx)$, $\mu$ and $\kappa$ satisfy the condition
\[
	\left(1 + \left[\mu^2\overline{\theta}/(1-\mu) + \mu\right] +   [\mu \overline{\theta}/(1-\mu)^2]  \right) \kappa \overline{\mathfrak{g}} < 1.
\]
Then the transport equation~\eqref{EQ:ERT TP} with boundary source ~\eqref{EQ:Collimated Source} has a unique solution $u$ such that $0\le \aver{u}\le \overline{\mathfrak g}$.
\end{theorem}
\begin{proof}
For any $m(\bx)\ge 0$, let $u$ be the solution to the following linear transport equation
\begin{equation}\label{EQ:MODIF ERT 2}
\begin{array}{rcll}
\bv \cdot \nabla u + (\sigma_a+\sigma_s) u(\bx, \bv) + \sigma_b(\bx) m(\bx) u(\bx, \bv) &=& \sigma_s(\bx) K u(\bx, \bv), &\mbox{in}\ X\, \\
u(\bx, \bv) &= & \mathfrak{g}(\bx) \delta(\bv - \bv'), &\mbox{on}\ \Gamma_{-}\,.
\end{array}
\end{equation}
We then define an operator $\cC: m \mapsto \aver{u}$ as in~\eqref{EQ:Operator C}, and introduce the following set of functions
\begin{equation}
\cM = \{m\in L^{\infty}(\Omega) \mid 0\le m(\bx) \le \cC(0) \},
\end{equation}
where $\cC(0)$ is the angularly averaged solution to~\eqref{EQ:MODIF ERT 2} with $m = 0$. It follows from linear transport theory that $\cC(\cM) \subseteq\cM$.

Let $u_1$ and $u_2$ be the solutions to~\eqref{EQ:MODIF ERT 2} with $m=m_1\in\cM$ and $m=m_2\in\cM$ respectively. Following the same notation as before, we define $\sigma_{t,i} := \sigma_a + \sigma_s + m_i\sigma_b$, $i=1,2$. We can then write the solutions $u_i$ ($1\le i\le 2$) as $u_i = u_{i, b} + u_{i, s}$ with
\begin{equation}\label{EQ: U2}
\begin{aligned}
u_{i,b}(\bx, \bv) &= \mathfrak{g}(\bx - \tau_{-}(\bx, \bv)\bv)\delta(\bv - \bv')\exp\left(-\int_{0}^{\tau_{-}(\bx , \bv)}\sigma_{t,i}(\bx - s\bv) ds\right)\,, \\
u_{i,s}(\bx, \bv) &= \int_{0}^{\tau_{-}(\bx, \bv)}\exp\left(-\int_0^l \sigma_{t,i}(\bx - s\bv) ds\right)(\sigma_s K u_i)(\bx - l\bv, \bv) dl\,.
\end{aligned}
\end{equation}
Following the definition of the operator $K$ in~\eqref{EQ:Scattering Op}, we have that
\begin{equation*}
K u_{i,b}(\bx', \bv) = \Theta(\bv, \bv')\mathfrak{g}(\bx' -\tau_{-}(\bx', \bv')\bv')\exp\left(-\int_0^{\tau_{-}(\bx', \bv')}\sigma_{t,i}(\bx'-s\bv') ds\right) \le \overline{\mathfrak{g}} \overline{\theta}.
\end{equation*}
where $\bx' := \bx - l\bv$, $l\in(0, \tau_{-}(\bx, \bv))$. Meanwhile, using the same procedure as in~\eqref{EQ:Contraction}, we have that,
\begin{multline*}
\|u_{i,s}\|_{L^{\infty}(X)} \le \left(1 - \exp\left[-\int_0^{\tau_{-}(\bx, \bv)} \sigma_{t,i}(\bx - s\bv) ds \right]\right)\sup_{\bx\in\Omega}\left|\frac{\sigma_s(\bx)}{\sigma_{t, i}(\bx)}\right|(\|u_{i,s}\|_{L^{\infty}(X)} + \overline{\mathfrak{g}}\overline{\theta})\\
\le\mu\|u_{i,s}\|_{L^{\infty}(X)} +\mu \overline{\mathfrak{g}}\overline{\theta}\,.
\end{multline*}
This implies that 
\begin{equation}\label{EQ:Bound uis}
	\|u_{i,s}\|_{L^{\infty}(X)}\le \mu \overline{\mathfrak{g}}\overline{\theta}/(1-\mu). 
\end{equation}

Let us now verify that $\phi := u_1 -u_2$ solves
\begin{equation*}
\begin{array}{rcll}
    \bv\cdot \nabla \phi + (\sigma_a + \sigma_s + \sigma_b m_1)\phi &=& \sigma_s K \phi + u_2\sigma_b(m_2 - m_1), &\mbox{in}\ X\, \\
\phi(\bx, \bv) &= & 0, &\mbox{on}\ \ \Gamma_{-}\,.
\end{array}
\end{equation*}
In the same manner, we write $\phi = \phi_b + \phi_s$ with $\phi_b$ and $\phi_s$ given as
\begin{equation*}
\begin{aligned}
\phi_b(\bx, \bv)&= \int_{0}^{\tau_{-}(\bx, \bv)} \exp\left(-\int_0^{l}\sigma_{t,1}(\bx - s\bv) ds\right)\left(\sigma_s K\phi_b+  u_{2,b}\sigma_b (m_2-m_1)\right)(\bx-l\bv, \bv)dl\,,\\
\phi_s(\bx, \bv)&= \int_{0}^{\tau_{-}(\bx, \bv)} \exp\left(-\int_0^{l}\sigma_{t,1}(\bx - s\bv) ds\right)\left(\sigma_s K\phi_s+  u_{2,s}\sigma_b (m_2-m_1)\right)(\bx-l\bv, \bv)dl\,.
\end{aligned}
\end{equation*}

Use the representation of $u_{2,b}$, we obtain $\phi_b$ in the following form,
\begin{equation*}
\phi_b(\bx, \bv) = \phi_{b,b}(\bx)\delta(\bv - \bv') + \phi_{b,s}(\bx, \bv)\,,
\end{equation*}
where $|\phi_{b,b}|\le \kappa \overline{\mathfrak{g}} \|m_1-m_2\|_{L^{\infty}(\Omega)}$ with $\kappa = \sup_{\bx\in\Omega}|\frac{\sigma_b}{\sigma_a + \sigma_s}|$, and
\begin{equation}\label{EQ:PHI BS}
\begin{aligned}
\phi_{b,s}(\bx, \bv) &= \int_{0}^{\tau_{-}(\bx, \bv)} \exp\left(-\int_0^{l}\sigma_{t,1}(\bx - s\bv) ds\right)\sigma_s K \phi_{b,s}(\bx-l\bv, \bv)dl \\
&+\Theta(\bv, \bv') \int_{0}^{\tau_{-}(\bx, \bv)} \exp\left(-\int_0^{l}\sigma_{t,1}(\bx - s\bv) ds\right)\sigma_s  \phi_{b,b}(\bx-l\bv)dl\,.\\
\end{aligned}
\end{equation}
Note that the second term on the right-hand-side is bounded by $\overline{\theta} \overline{\mathfrak{g}}\kappa \mu \|m_1-m_2\|_{L^{\infty}(\Omega)}$. Therefore, we have
\begin{equation}\label{EQ:phibs}
\|\phi_{b,s}\|_{L^{\infty}(X)}\le \overline{\theta} \overline{\mathfrak{g}}\kappa \mu \|m_1-m_2\|_{L^{\infty}(\Omega)}/(1-\mu)\,.
\end{equation}
Integrating~\eqref{EQ:PHI BS} over $\bbS^{d-1}$, and then using~\eqref{EQ:phibs}, we have,
\begin{equation}\label{EQ:phibs2}
|\aver{\phi_{b,s}}| \le \mu \|\phi_{b,s}\|_{L^{\infty}(X)} + \mu\kappa\overline{\mathfrak{g}} \|m_1-m_2\|_{L^{\infty}(\Omega)} \le \left[\mu^2\overline{\theta}/(1-\mu) + \mu\right] \kappa \overline{\mathfrak{g}} \|m_1-m_2\|_{L^{\infty}(\Omega)}\,.
\end{equation}
In a similar manner, we can estimate, using~\eqref{EQ:Bound uis},
\begin{equation}\label{EQ:phis}
\|\phi_s\|_{L^{\infty}(X)} \le \kappa \|u_{2,s}\|_{L^{\infty}(X)} \|m_1-m_2\|_{L^{\infty}(\Omega)} /(1-\mu) 
\le  [\mu \overline{\theta}/(1-\mu)^2] \kappa \overline{\mathfrak{g}}  \|m_1-m_2\|_{L^{\infty}(\Omega)}\,.
\end{equation}

The bounds in~\eqref{EQ:phibs2} and ~\eqref{EQ:phis} now allow us to have
\begin{equation*}
\begin{aligned}
|\aver{\phi}|&\le |\aver{\phi_{b}}| + |\aver{\phi_{s}}|\\
&\le |\phi_{b,b}| + |\aver{\phi_{b,s}}| + \|{\phi_{s}}\|_{L^{\infty}(X)}\\
&\le \left(1 + \left[\mu^2\overline{\theta}/(1-\mu) + \mu\right] +   [\mu \overline{\theta}/(1-\mu)^2]  \right) \kappa \overline{\mathfrak{g}}\|m_1-m_2\|_{L^{\infty}(\Omega)}.
\end{aligned}
\end{equation*}
When the constant $(1 + \left[\mu^2\overline{\theta}/(1-\mu) + \mu\right] +   [\mu \overline{\theta}/(1-\mu)^2]) \kappa \overline{\mathfrak{g}} < 1$, the mapping $\cC$ is a contraction in $L^{\infty}(\Omega)$ norm, the Banach Fixed-Point Theorem~\cite{Zeidler-Book95} implies that the solution is unique in $\cM$.
\end{proof}
\begin{remark}\rm
    Unlike in the previous section, we are not able to use the Schauder Fixed-Point Theorem in this proof due to the lack of compactness of the map $\cC$. One can make additional assumptions on the smoothness of all the coefficients involved as well as the physical domain $\Omega$ to recover such compactness. We did not pursue in this direction.
\end{remark}

%%%%%%%%%%%%%%%%%%%%%%%%%%%%%%%%%%%%%%%%%%%%%%%%%%%%%%%%%%%%%%%%%%
%%%%%%%%%%%%%%%%%%%%%%%%%%%%%%%%%%%%%%%%%%%%%%%%%%%%%%%%%%%%%%%%%%
\section{Inversion in non-scattering media}
\label{SEC:IP No Scattering}
%%%%%%%%%%%%%%%%%%%%%%%%%%%%%%%%%%%%%%%%%%%%%%%%%%%%%%%%%%%%%%%%%%
%%%%%%%%%%%%%%%%%%%%%%%%%%%%%%%%%%%%%%%%%%%%%%%%%%%%%%%%%%%%%%%%%%

We start with the case of non-scattering media where $\sigma_s(\bx)\equiv 0$. In this case, the original transport model~\eqref{EQ:ERT TP} is simplified into a free transport equation which is essentially an ordinary differential equation parameterized by the angular variable $\bv$. We can obtain an explicit method for the reconstructions with either collimated sources or point sources. Similar analysis for the linear transport equation can be found in~\cite{MaRe-CMS14}. Inversion in this setup with a general bounded source will be treated in the next section as a special case.

%%%%%%%%%%%%%%%%%%%%%%%%%%%%%%%%%%%%%%%%%%%%%%%%%%%%%%%%%%%%%%%%%%
\subsection{Inversion with collimated sources}
%%%%%%%%%%%%%%%%%%%%%%%%%%%%%%%%%%%%%%%%%%%%%%%%%%%%%%%%%%%%%%%%%%

With collimated sources, we can integrate the transport equation along direction $\bv$ to get the following integral representation of the transport solution, when $\sigma_s\equiv 0$:
\begin{equation}\label{EQ: NON SCATTER SOL}
	u(\bx, \bv) = g(\bx- \tau_{-}(\bx, \bv)\bv, \bv) \exp\left( - \int_0^{\tau_{-}(\bx, \bv)} (\sigma_a(\bx - s\bv) + \sigma_b\aver{u}(\bx - s\bv) )ds\right)\,.
\end{equation}

Let us assume that we have data generated from two collimated sources, $g_j(\bx, \bv) = \mathfrak{g}_j(\bx)\delta(\bv - \bv')$ ($j=1,2$), focused in the same direction $\bv'\in\bbS^{d-1}$ but with different strengths $\mathfrak{g}_1\neq \mathfrak{g}_2$.  Then the corresponding data are:
\begin{equation*}
H_j(\bx) =  \sigma_a(\bx)\aver{u_j}(\bx) +\sigma_b(\bx) \aver{u_j}^2(\bx)\,, \quad j=1,\ 2
\end{equation*}
with $u_j$ satisfying
\begin{equation*}
	u_j(\bx, \bv) = \mathfrak{g}_j(\bx- \tau_{-}(\bx, \bv)\bv)\delta(\bv - \bv') \exp\left( - \int_0^{\tau_{-}(\bx, \bv)} (\sigma_a(\bx - s\bv) + \sigma_b\aver{u_j}(\bx - s\bv) )ds\right)\,.
\end{equation*}
We can integrate $u_j(\bx, \bv)$ over $\bv\in \bbS^{d-1}$ to get,
\begin{equation}\label{EQ:uj average}
\aver{u_j}(\bx) = \mathfrak{g}_j(\bx - \tau_{-}(\bx, \bv'),\bv')\exp\left( - \int_0^{\tau_{-}(\bx, \bv')} (\sigma_a(\bx - s\bv') + \sigma_b\aver{u_j}(\bx - s\bv') )ds\right)\,.
\end{equation}
For any fixed $\bx\in\Omega$, we introduce the notations $\bx' := \bx - \tau_{-}(\bx, \bv') \bv'\in\partial\Omega$, $\phi_j(s) := \aver{u_j}(\bx' + s\bv')$ , $\wt \sigma_a(s):= \sigma_a(\bx' + s\bv')$ and $\wt \sigma_b(s) := \sigma_b(\bx' +  s\bv')$. Then ~\eqref{EQ:uj average} is equivalent to:
\begin{equation}\label{EQ: TPA NONSCAT INT}
\phi_j(t) = \mathfrak{g}_j(\bx')\exp\left(-\int_0^t (\wt \sigma_a(s) + \wt \sigma_b(s)\phi_j(s) )ds\right)
\end{equation}
with $\phi_j(0) = \mathfrak{g}_j(\bx')$. Taking the logarithm of both sides of~\eqref{EQ: TPA NONSCAT INT} and then differentiate with respect to $t$, we obtain the following ODE for $\phi_j(t)$,
\begin{equation}\label{EQ: TPA NONSCAT ODE}
\phi_j'(t) = -(\wt \sigma_a(t) +\wt \sigma_b(t)\phi_j(t)) \phi_j(t),\quad t\in [0, \tau_{-}(\bx, \bv')]\,.
\end{equation}
From the definition of the internal data~\eqref{EQ:Data TP}, we notice that the right-hand-side of~\eqref{EQ: TPA NONSCAT ODE} is exactly $-H_j(\bx' + t\bv')$. We can therefore reconstruct $\phi_j(t)$ from the datum $H_j$ as
\begin{equation}\label{EQ:Rec Sol}
	\wt \phi_j(t) = \mathfrak{g}_j(\bx') -\int_{0}^t H_j(\bx' + s\bv') ds \,.
\end{equation}

Once we reconstructed $\{\wt\phi_j(t)\}_{j=1}^2$, we can reconstruct $\wt\sigma_a$ and $\wt \sigma_b$ from the data by solving the following system of equations at any $t\in[0, \tau_{-}(\bx,\bv')]$:
\begin{equation}\label{EQ: LINEAR EQS}
	\begin{array}{rcl}
	\wt \sigma_a(t) \wt\phi_1(t) + \wt\sigma_b(t) \wt\phi^2_1(t) &=& H_1(\bx' + t\bv'),\\
	\wt\sigma_a(t) \wt\phi_2(t) + \wt \sigma_b(t) \wt\phi^2_2(t) &=& H_2(\bx' + t\bv').
	\end{array}
\end{equation}
This linear system, for the unknown coefficient pair $(\sigma_a, \sigma_b)$, is uniquely invertible at $t\in [0, \tau_{-}(\bx, \bv')]$ if $\phi_1(t)\neq \phi_2(t)$. 

The following result shows that if the data $\{H_j\}_{j=1}^2$ are consistent with the model, that is, if the data are generated from the model with the true coefficients, then we can select the illumination sources $\mathfrak{g}_1$ and $\mathfrak{g}_2$ to be sufficiently different to make the system~\eqref{EQ: LINEAR EQS} invertible.
\begin{lemma}\label{LEM: MONO}
    If $\mathfrak{g}_1 > \mathfrak{g}_2 > 0$, then $\phi_1(t) > \phi_2(t)$, $\forall t\in [0, \tau_{-}(\bx, \bv')]$.
\end{lemma}
\begin{proof}
    From ~\eqref{EQ: TPA NONSCAT INT} and the non-negativity of transport solutions, we conclude that $\mathfrak{g}_j > 0$ implies $\mathfrak{g}_j \ge \phi_j(t) > 0$. We check that $z(t) := \phi_1(t) - \phi_2(t)$ satisfies
    \begin{equation}
    \frac{z'(t)}{z(t)} = -(\wt \sigma_a(t) +\wt\sigma_b(t)(\phi_1(t)+\phi_2(t))) \,.
    \end{equation}
    This implies that
    \begin{equation}
    z(t) = z(0) \exp\left( -\int_0^t (\wt\sigma_a(s) +\wt\sigma_b(s)(\phi_1(s)+\phi_2(s))) ds  \right).
    \end{equation}
    We then conclude that $z(t) > 0$ using the assumption that $z(0) =\phi_1(0) - \phi_2(0) > 0$.
\end{proof}

To summarize, in order to reconstruct the coefficients $\sigma_a$ and $\sigma_b$, we first reconstruct the solutions~\eqref{EQ:Rec Sol} from the data. We then solve the linear system~\eqref{EQ: LINEAR EQS} to reconstruct $(\sigma_a, \sigma_b)$.

%%%%%%%%%%%%%%%%%%%%%%%%%%%%%%%%%%%%%%%%%%%%%%%%%%%%%%%%%%%%%%%%%%
\subsection{Inversion with point sources}
%%%%%%%%%%%%%%%%%%%%%%%%%%%%%%%%%%%%%%%%%%%%%%%%%%%%%%%%%%%%%%%%%%

An explicit reconstruction method can also be developed in the case when point sources are used to illuminate the media. Let $
g_j(\bx, \bv) = \mathfrak{g}_j(\bv)\delta(\bx - \bx')$ ($j=1, 2$) with $\mathfrak{g}_1\neq \mathfrak{g}_2$ positive constants. To be technically correct in the derivation below, we assume that $\sigma_b$ vanishes in the vicinity of $\bx'\in\partial\Omega$, that is, $\sigma_b \equiv 0$ in $B_{\eps}(\bx')\cap \Omega$ for some $\eps > 0$. In applications, this can be done in a straightforward way by placing the illuminating point source a little away from the surface of the media (which, mathematically, is equivalent to extending the domain $\Omega$ to a slightly larger domain $\Omega'$ with $\sigma_b\equiv 0$ in $\Omega'\backslash \bar\Omega$ ). We can then integrate the transport equation along the direction of each ray out of the point source to have
\begin{equation}\label{EQ:Sol Point S}
\aver{u_j}(\bx) = \mathfrak{g}_j (\bv)\left|\bn(\bx')\cdot\bv\right| \frac{\exp\left(-\dint_0^{\tau_{-}(\bx, \bv)}(\sigma_a+\sigma_b\aver{u_j})(\bx - s\bv) ds\right)}{|\bx - \bx'|^{d-1}},
\end{equation}
where $\bv = \frac{\bx - \bx'}{|\bx - \bx'|}$. The parameterization of the line segment between $\bx$ and $\bx'$ is the same as before: $\{\bx'+s\bv \mid s\in(0, |\bx-\bx'|)\}$. Let $\phi_j(s) := \aver{u_j}(\bx' + s\bv)$, $\wt \sigma_a(s) := \sigma_a(\bx' + s\bv)$, $\wt \sigma_b(s) := \sigma_b(\bx' + s\bv)$, and $\wt H_j(s) := H_j(\bx' + s\bv)$, then we can write~\eqref{EQ:Sol Point S} as
\begin{equation*}
	\phi_j(t) = \mathfrak{g}_j (\bv) \left|\bn(\bx')\cdot\bv\right| t^{1-d} \exp\left(-\int_{0}^t (\wt \sigma_a + \wt \sigma_b\phi_j)(s)ds\right).
\end{equation*}
Taking the derivative with respect to $t$, we obtain that 
\begin{equation}
	\phi_j'(t)  = \dfrac{1-d}{t}\phi_j(t) - \big(\wt \sigma_a+\wt \sigma_b\phi_j\big)\phi_j(t) .
\end{equation}
We can then replace $\big(\wt \sigma_a+\wt \sigma_b\phi_j\big)\phi_j(t)$ in the equation with the data $H_j$ and integrate the ODE, using the asymptotic behavior of $\phi_j(t)$ as $t \to 0$ from \eqref{EQ:Sol Point S}, to reconstruct the solution $\phi_j$:
\begin{equation}
	\wt \phi_j(t)  = \frac{1}{t^{d-1}}\left(\mathfrak{g}_j(\bv)|\bn(\bx')\cdot \bv|- \int_0^t \wt{H}_j(s) s^{d-1} ds\right)\,.
\end{equation}
The remaining task is to reconstruct $\wt{\sigma}_a$ and $\wt{\sigma}_b$ from the system of equations:
\begin{equation}
	\wt{\sigma}_a(t) \phi_j(t) + \wt{\sigma}_b(t) \phi_j^2(t) = \wt{H}_j(t),\quad j=1, 2.
\end{equation}
Use the similar argument as in Lemma~\ref{LEM: MONO}, it can be shown that $0 < \mathfrak{g}_1 < \mathfrak{g}_2$ is sufficient to ensure uniqueness of the inversion when the corresponding data are consistent with the model.

%%%%%%%%%%%%%%%%%%%%%%%%%%%%%%%%%%%%%%%%%%%%%%%%%%%%%%%%%%%%%%%%%%
%%%%%%%%%%%%%%%%%%%%%%%%%%%%%%%%%%%%%%%%%%%%%%%%%%%%%%%%%%%%%%%%%%
\section{Inversion in media with known scattering}
\label{SEC:IP Scattering}
%%%%%%%%%%%%%%%%%%%%%%%%%%%%%%%%%%%%%%%%%%%%%%%%%%%%%%%%%%%%%%%%%%
%%%%%%%%%%%%%%%%%%%%%%%%%%%%%%%%%%%%%%%%%%%%%%%%%%%%%%%%%%%%%%%%%%

We now study the inverse problem of reconstructing the absorption coefficients $\sigma_a$ and $\sigma_b$ from data $H$ in scattering media with the scattering coefficient $\sigma_s$ assumed known.

\subsection{Stability of inversion}
%%%%%%%%%%%%%%%%%%%%%%%%%%%%%%%%%%%%%%%%%%%%%%%%%%%%%%%%%%%%%%%%%%

We start with the inverse problem of reconstructing the absorption coefficient $\Sigma_a$ in the linear transport equation~\eqref{EQ:ERT} from internal data set of the form~\eqref{EQ:Data}.

%We limit ourselves to a smaller set of coefficients and data than what is specified in $(\cA)$. 
Let $H$ be the internal datum~\eqref{EQ:Data} generated from the linear transport model~\eqref{EQ:ERT} with the absorption coefficient $\Sigma_a\in \cF_{\underline\Sigma_a}^{\overline\Sigma_a}(\Omega)$ and the boundary source $g$. For a given $\alpha>0$, we define the set
\[
	\Pi_{\alpha}:=
\left\{(\Sigma_a, H, g) \mid \Sigma_a-\frac{\bv\cdot \nabla \Sigma_a}{\Sigma_a}+\frac{\bv\cdot\nabla H}{H} \ge \alpha>0, \forall (\bv, \bx)\in X\right\}.
\]
Using the fact that $H=\Sigma_a \aver{u}$, $u$ being the solution to~\eqref{EQ:ERT} with coefficient $\Sigma_a$ and source $g$, we see that $\Pi_\alpha$ is equivalent to
\[
	\Pi_{\alpha}':=
\left\{(\Sigma_a, H, g) \mid \Sigma_a+\frac{\bv\cdot \nabla \aver{u}}{\aver{u}}\ge\alpha>0, \forall (\bv, \bx)\in X\right\}.
\]
We show next that we could stably reconstruct coefficients and data combinations $(\Sigma_a, H, g)$ in the class of $\Pi_\alpha$.

%We first prove the following result.
\begin{theorem}\label{THM:Stab Lin ERT}
Let $H$ and $\wt H$ be two data sets generated with coefficients $\Sigma_a$ and $\wt\Sigma_a$ respectively from~\eqref{EQ:ERT} in the form of~\eqref{EQ:Data} with boundary source $g$. %Assume that $(\Sigma_a, H, g)$ and $(\wt\Sigma_a, \wt H, g)$ satisfy $(\cA)$, $(\cA')$ and $(\cA'')$. 
Assume that there exists constants $\alpha>0$ and $1>\beta>0$ such that:\\[0.9ex]
$(\cA')$ (i) $(\Sigma_a, H, g),\, (\wt\Sigma_a, \wt H,  g) \in \Pi_{\alpha}$,\ \ and,\ \ (ii) ${\Sigma_a}_{|\partial\Omega}$ is know and $\overline\Sigma_a\|\dfrac{g H}{{\Sigma_a}_{|\partial\Omega}}\|_{L^\infty(\Gamma_-)}\le \beta$.\\%[0.5ex]
Then, under the assumtions in $(\cA)$, the following stability holds for some constants $\fc, \wt \fc>0$:
\begin{equation}\label{EQ:Stab sigma}
	\wt\fc \|H-\wt H\|_{L^2(\Omega)} \le \|\Sigma_a-\wt \Sigma_a\|_{L^2(\Omega)} \le \fc \|H-\wt H\|_{L^2(\Omega)}.
\end{equation}
\end{theorem}
\begin{proof}
Let $u$ and $\wt u$ be solutions to the transport equation~\eqref{EQ:ERT} with coefficients $\Sigma_a$ and $\tilde\Sigma_a$ respectively. By Lemma~\ref{LEM:Positivity Lin ERT}, we have that $\overline{g} \ge u, \wt u\ge \eps>0$ for some $\eps$. 

Let $w:=u-\wt u$. Then we check that
\begin{equation*}
H-\wt H=\wt \Sigma_a \aver{w} +(\Sigma_a-\wt \Sigma_a)\aver{u}.
\end{equation*}
This leads to the following equality:
\begin{equation}\label{EQ:Data TP Diff}
\dfrac{u}{\aver{u}}(H-\wt H)=\wt \Sigma_a\dfrac{u}{\aver{u}} \aver{w} + (\Sigma_a-\wt \Sigma_a)u.
\end{equation}
Therefore, we have that,
\begin{equation}\label{EQ:Bound A0}
	\|\dfrac{u}{\aver{u}}(H-\wt H)\|_{L^2(X)}\le \| \wt \Sigma_a\dfrac{u}{\aver{u}} \aver{w}\|_{L^2(X)}+\|(\Sigma_a-\wt \Sigma_a)u\|_{L^2(X)}.
\end{equation}
We also observe that $w\in L^\infty(X)$ solves the following transport equation:
\begin{equation*}
	\begin{array}{rcll}
  	\bv\cdot\nabla w+ (\wt \Sigma_{a}+\sigma_{s}) w & = & \sigma_{s} K w(\bx, \bv)-(\Sigma_a-\wt \Sigma_{a}) u, &\mbox{in}\ X\\
    w(\bx,\bv) &=& 0, & \mbox{on}\ \Gamma_-.
	\end{array}
\end{equation*}
We therefore deduce, from the standard transport theory~\cite{DaLi-Book93-6}, that
\begin{equation}\label{EQ:Bound B0}
	\|w\|_{L^2(X)}\le \|(\Sigma_a-\wt \Sigma_a)u\|_{L^2(X)}.
\end{equation}
The left-hand-side of~\eqref{EQ:Stab sigma} then follows from~\eqref{EQ:Bound A0} and~\eqref{EQ:Bound B0}, together with the boundedness of the coefficients and the corresponding solutions as well as the fact that $\|\aver{w}\|_{L^2(\Omega)} \le \|w\|_{L^2(X)}$.

Meanwhile,~\eqref{EQ:Data TP Diff} also implies that
\begin{multline}\label{EQ:Bound A}
	\|(\Sigma_a-\wt \Sigma_a)u\|_{L^2(X)}\le \|\dfrac{u}{\aver{u}}(H-\wt H)\|_{L^2(X)}+\|\wt\Sigma_a\dfrac{u}{\aver{u}} \aver{w}\|_{L^2(X)}\\
	\le \overline\Sigma_a \|\dfrac{u}{\aver{u}}\|_{L^\infty(X)}\Big(\|\dfrac{H-\wt H}{\wt\Sigma_a}\|_{L^2(\Omega)}+\| w\|_{L^2(X)}\Big) \\ 
	\le \overline\Sigma_a \|\dfrac{u}{\aver{u}}\|_{L^\infty(X)}\Big(\|\dfrac{H-\wt H}{\wt\Sigma_a}\|_{L^2(\Omega)}+\|(\Sigma_a-\wt \Sigma_a)u\|_{L^2(X)}\Big),
\end{multline}
where the last step comes from~\eqref{EQ:Bound B0}.

%Let $\phi:=\wt\sigma_a\dfrac{u}{\aver{u}}$. Then some simple algebra shows that $\phi$ solves the transport equation:
%\begin{equation*}
%	\begin{array}{rcll}
%  	\bv\cdot\nabla \phi+(\sigma_{a}+\sigma_{s}) \phi & = & \sigma_{s} K \phi +(\bv\cdot\nabla\ln\dfrac{\wt\sigma_a}{\aver{u}}) \phi, &\mbox{in}\ X\\
%    \phi(\bx,\bv) &=& \wt\sigma_a \dfrac{g}{\aver{\phi_{|\partial\Omega}}}, & \mbox{on}\  \Gamma_-.
%	\end{array}
%\end{equation*}
Let $\phi:=\dfrac{u}{\aver{u}}$. Then some simple algebra shows that $\phi$ solves the transport equation:
\begin{equation*}
	\begin{array}{rcll}
  	\bv\cdot\nabla \phi+(\Sigma_a+\bv\cdot\nabla\ln\aver{u}+\sigma_{s}) \phi & = & \sigma_{s} K \phi, &\mbox{in}\ X\\
    \phi(\bx,\bv) &=& \dfrac{g H}{{\Sigma_a}_{|\partial\Omega}}, & \mbox{on}\  \Gamma_-
	\end{array}
\end{equation*}
where the boundary condition comes from the assumption that ${\Sigma_{a}}_{|\partial\Omega}$ is know (which implies that $\aver{u}_{|\partial\Omega}=\dfrac{H}{{\Sigma_{a}}_{|\partial\Omega}}$). The first assumption in $(\cA')$ means that $\Sigma_a+\bv\cdot\nabla\ln\aver{u} \ge \alpha>0$. Therefore, we can use the maximum principle, ensured by the assumption on the scattering kernel $\Theta$ in $(\cA)$, to conclude that
\begin{equation}\label{EQ:Bound phi}
	\|\phi\|_{L^\infty(X)}\le \| \dfrac{g H}{{\Sigma_a}_{|\partial\Omega}}\|_{L^\infty(\Gamma_-)}. %+\dfrac{1}{\underline\sigma_a}\|\bv\cdot\nabla\ln\dfrac{\wt\sigma_a}{\aver{u}}\|_{L^\infty(X)}\|\phi\|_{L^\infty(X)}.
\end{equation}
The bound in~\eqref{EQ:Bound A} then implies that
\begin{equation*} %\label{EQ:Bound A}
	\|(\Sigma_a-\wt \Sigma_a)u\|_{L^2(X)}\le \overline \Sigma_a \| \dfrac{g H}{{\Sigma_a}_{|\partial\Omega}}\|_{L^\infty(\Gamma_-)}\Big(\|\dfrac{H-\wt H}{\wt\Sigma_a}\|_{L^2(\Omega)}+\|(\Sigma_a-\wt \Sigma_a)u\|_{L^2(X)}\Big).
\end{equation*}
%The first assumption in $(\cA')$ implies that
%\[
%	\dfrac{1}{\underline\sigma_a}\|\bv\cdot\nabla\ln\dfrac{\wt\sigma_a}{\aver{u}}\|_{L^\infty(X)}\le \beta_1<1.
%\]
%Therefore, ~\eqref{EQ:Bound phi}, together with the assumption that $g\equiv g(\bx)$, implies that
%\begin{equation*}
%\|\phi\|_{L^\infty(X)}\le \dfrac{\|\wt\sigma_a \dfrac{g}{\aver{u_{|\partial\Omega}}}\|_{L^\infty(X)}}{1-\beta_1}=\dfrac{\overline\sigma_a}{1-\beta_1}.
%\end{equation*}
This bound, together with the second assumption in $(\cA')$, then implies that
\begin{equation*}
	\|(\Sigma_a-\wt \Sigma_a)u\|_{L^2(X)}\le \dfrac{\beta}{1-\beta} \|\dfrac{H-\wt H}{\wt\Sigma_a}\|_{L^2(\Omega)}.
\end{equation*}
This gives the right-hand-side of the stability bound~\eqref{EQ:Stab sigma}.
\end{proof}

The above theorem shows that, in appropriate settings, the absorption coefficient in the transport equation can be reconstructed stably with one interior datum $H$. This means that if we think of the term $\sigma_a+\sigma_b\aver{u}$ in the semilinear transport equation~\eqref{EQ:ERT TP} as a single absorption coefficient, we can reconstruct this coefficient from a single data. This simple idea leads to a method to reconstruct $\sigma_a$ and $\sigma_b$ from two data sets. We now describe the method.

We will need the following result.
\begin{lemma}\label{LEM:Src Selection}
	For a given set of $(\sigma_a, \sigma_b)\in\cF_{\underline\sigma_a}^{\overline\sigma_a}(\Omega)\times \cF_{\underline\sigma_b}^{\overline\sigma_b}(\Omega)$ and $(\sigma_s, \Theta)$ satisfying $(\cA)$, there exist two boundary sources $g_1$ and $g_2$ given as in~\eqref{EQ:SOURCE} such that:
	\[
		|\aver{u_1}-\aver{u_2}| \ge \eps, \quad \mbox{for some}\ \eps>0
	\]
where $u_1$ and $u_2$ are solutions to ~\eqref{EQ:ERT TP} with $g_1$ and $g_2$ respectively.
\end{lemma}
\begin{proof}
Let $g$ be such that 
\begin{equation}\label{EQ:Source Sel}
\gamma \min\Big(\dfrac{\underline \sigma_a}{\overline\sigma_b},\ \underline{\theta}  \frac{\underline\sigma_s(\bx)}{\overline \sigma_b(\bx)}\Big) \ge \overline{g}, \quad\mbox{and}\quad \underline{g}>0,\ \mbox{for some}\ 0<\gamma<1.
\end{equation}
It is clear that $g$ satisify ~\eqref{EQ:SOURCE}. Let $g_1\neq g_2$ be given as in~\eqref{EQ:Source Sel}. Following Corollary~\ref{COR:Positivity}, we have $0<\eps'\le u_1\le g_1$ and $0<\eps' \le u_2\le g_2$ for some $\eps'>0$. 
Let $w:=u_1-u_2$. Then $w$ solves
\begin{equation}\label{EQ:ERT Max}
\begin{array}{rcll}
\bv \cdot \nabla w(\bx,\bv) +\Sigma w  &=& \dint_{\bbS^{d-1}}\Sigma_s(\bx, \bv, \bv') w, &\mbox{in}\ X\, \\
w(\bx,\bv) &= & g_1-g_2, &\mbox{on}\  \Gamma_{-}\,.
\end{array}
\end{equation}
where $\Sigma(\bx, \bv):=\sigma_a+\sigma_b\frac{\aver{u_1}+\aver{u_2}}{2}+\sigma_s$, $\Sigma_s(\bx, \bv, \bv'):=\sigma_s(\bx)\Theta(\bv, \bv')-\sigma_b\frac{u_1+u_2}{2}$. With the assumptions in $(\cA)$ and the fact that $g_1$ and $g_2$ satisfying~\eqref{EQ:Source Sel}, we can verify that $\sigma_s\overline\theta \ge \Sigma_s(\bx, \bv, \bv')\ge (1-\gamma)\underline\theta \underline\sigma_s>0$ and $\Sigma-\dint_{\bbS^{d-1}}\Sigma_s(\bx, \bv, \bv')d\bv'\ge (1-\gamma)\underline\sigma_a>0$ (where $\gamma$ is given in ~\eqref{EQ:Source Sel}). Therefore, the solution to~\eqref{EQ:ERT Max} satisfies the maximum principle. By selecting $g_1-g_2\ge \eps''$ for some $\eps''>0$, we have that $w\ge \eps$ for some $\eps>0$ using Lemma~\ref{LEM:Positivity Lin ERT}.
\end{proof}

Theorem~\ref{THM:Stab Lin ERT} allows us to estimate the stability of reconstructing $(\sigma_a, \sigma_b)$.
\begin{corollary}\label{COR:Stab Semil ERT}
Let $(\sigma_a, \sigma_b)\in\cF_{\underline\sigma_a}^{\overline\sigma_a}(\Omega)\times \cF_{\underline\sigma_b}^{\overline\sigma_b}(\Omega)$ and $(\wt \sigma_a, \wt\sigma_b)\in\cF_{\underline\sigma_a}^{\overline\sigma_a}(\Omega)\times \cF_{\underline\sigma_b}^{\overline\sigma_b}(\Omega)$ be two sets of absorption coefficients, and $\bH:=(H_1, H_2)$ and  $\wt \bH:=(\wt H_1, \wt H_2)$ the corresponding data generated with $\bg=(g_1, g_2)$. Assume that $\bg$ is selected as in Lemma~\ref{LEM:Src Selection}, and $(\Sigma_a^j:=\sigma_a+\sigma_b\aver{u_j}, H_j, g_j)$ and $(\wt\Sigma_a^j:=\sigma_a+\wt\sigma_b\aver{\wt u_j}, \wt H_j, \wt g_j)$ $(j=1, 2)$ satisfy $(\cA')$. Then, under $(\cA)$, there exists constants $\fc, \wt\fc>0$ such that
\begin{equation}\label{EQ:Stab sigma ab}
\wt\fc
\|\bH-\wt\bH \|_{L^2(\Omega)}
\le
\left\|\begin{pmatrix}
\sigma_a \\
\sigma_b 
\end{pmatrix} - \begin{pmatrix}
\wt \sigma_a \\
\wt\sigma_b 
\end{pmatrix}\right\|_{L^2(\Omega)} 
\le 
\fc \|\bH-\wt\bH \|_{L^2(\Omega)}.
\end{equation}
\end{corollary}
\begin{proof}
The left inequality can be derived in the same manner as in Theorem~\ref{THM:Stab Lin ERT}. %Let $u_j$ and $\wt u_j$ ($j=1,2$) be solutions to the transport equation~\eqref{EQ:ERT} with parameters $(\sigma_a, \sigma_b, g_j)$ and $(\tilde\sigma_a, \wt\sigma_b, g_j)$ respectively. 
%We first note that b the assumptions, we have that $\overline{g} \ge u_j, \wt u_j\ge \eps>0$ for some $\eps$. 
%Let $\Sigma_j:=\sigma_a+\sigma_b\aver{u_j}$ and $\wt\Sigma_j:=\wt\sigma_a+\wt\sigma_b\aver{\wt u_j}$, $j=1,2$.  
We define $w_j:=u_j-\wt u_j$. Then some straightforward algebra leads us to the fact that
\begin{equation*}
H_j-\wt H_j=\Big(\wt\Sigma_a^j+\wt\sigma_b\aver{u_j}\Big) \aver{w_j}+ \Big[(\sigma_a-\wt \sigma_a)+(\sigma_b-\wt\sigma_b)\aver{u_j}\Big]\aver{u_j}.
\end{equation*}
With the boundedness of the coefficients as well as the solutions, we conclude that
\begin{equation}\label{EQ:Bound A0 2}
	\| H_j-\wt H_j \|_{L^2(\Omega)}\le \fc_1' \| w_j \|_{L^2(\Omega)}+\fc'_2\|(\sigma_a-\wt \sigma_a)+(\sigma_b-\wt\sigma_b)\aver{u_j}\|_{L^2(\Omega)}.
\end{equation}
The next step is to verify that $w_j$ solves the linear transport equation:
\begin{equation*}
	\begin{array}{rcll}
  	\bv\cdot\nabla w_j+ (\wt \Sigma_a^j+\sigma_{s}) w_j & = & \wt K w_j  - \Big[(\sigma_a-\wt \sigma_{a})+(\sigma_b-\wt\sigma_b)\aver{u_j}\Big]u_j, &\mbox{in}\ X\\
    w_j(\bx,\bv) &=& 0, & \mbox{on}\ \Gamma_-
	\end{array}
\end{equation*}
with the scattering operator $\wt K$ defined as
\[
	\wt K w_j : = \int_{\bbS^{d-1}}\Big[\sigma_s\Theta-\wt\sigma_b u_j\Big] w_j(\bx, \bv') d\bv' .
\]
Following the same argument as in Lemma~\ref{LEM:Src Selection}, this transport equation is uniquely invertible with a stability bound
\begin{equation}\label{EQ:Bound B0 2}
	\|w_j\|_{L^2(X)}\le \fc_3'\|(\sigma_a-\wt \sigma_{a})+(\sigma_b-\wt\sigma_b)\aver{u_j}\|_{L^2(X)}.
\end{equation}
Therefore, we have, from~\eqref{EQ:Bound A0 2} and~\eqref{EQ:Bound B0 2}, that
\begin{equation}\label{EQ:Bound C0 2}
	\|H_j-\wt H_j\|_{L^2(\Omega)}\le\fc_4' \|(\sigma_a-\wt \sigma_{a})+(\sigma_b-\wt\sigma_b)\aver{u_j}\|_{L^2(X)}.
\end{equation}
With the selection of $g_1$ and $g_2$, we conclude from Lemma~\ref{LEM:Src Selection} that the matrix
\[
	P:=\begin{pmatrix} 1 & \aver{u_1}\\ 1 & \aver{u_2} \end{pmatrix}
\]
is invertible with a bounded inverse at every point $\bx\in\Omega$. The left-hand-side of~\eqref{EQ:Stab sigma ab} then follows this fact and ~\eqref{EQ:Bound C0 2}.

To get the second bound in~\eqref{EQ:Stab sigma ab}, we notice that, by Theorem~\ref{THM:Stab Lin ERT} (which requires the assumptions we have made), we have the bound
\begin{equation*}
	\|\Sigma_a^j-\wt \Sigma_a^j\|_{L^2(\Omega)} \le \fc' \|H_j-\wt H_j\|_{L^2(\Omega)}
\end{equation*}
for some constant $\fc'>0$. This gives that
\begin{equation}\label{EQ:Bound D0 1}
	\left\|\begin{pmatrix}
\Sigma_a^1 \\
\Sigma_a^2 
\end{pmatrix} - \begin{pmatrix}
\wt \Sigma_a^1 \\
\wt\Sigma_a^2 
\end{pmatrix}\right\|_{L^2(\Omega)} \le \fc'' \|\bH-\wt \bH\|_{L^2(\Omega)}
\end{equation}
for some constant $\fc''>0$. Meanwhile, we verify that
\begin{equation*}
	\begin{pmatrix}
	\Sigma_a^1-\wt\Sigma_a^1 \\
	\Sigma_a^2-\wt\Sigma_a^2
	\end{pmatrix} 
	=
	\begin{pmatrix}
	1 & \aver{u_1}\\
	1 & \aver{u_2}
	\end{pmatrix}
	\begin{pmatrix}
	\sigma_a-\wt\sigma_a\\
	\sigma_b-\wt\sigma_b
	\end{pmatrix}
	-
	\begin{pmatrix}
	\dfrac{u_1}{\wt\Sigma_a^1}(\Sigma_a^1-\wt\Sigma_a^1)\\
	\dfrac{u_2}{\wt\Sigma_a^2}(\Sigma_a^2-\wt\Sigma_a^2)
	\end{pmatrix}
	+
	\begin{pmatrix}
		\dfrac{\wt\sigma_b}{\Sigma_a^1}(H_1-\wt H_1)\\
		\dfrac{\wt\sigma_b}{\Sigma_a^2}(H_2-\wt H_2)
	\end{pmatrix}.
\end{equation*}
This leads, using again the fact that the matrix $P$ has a bounded inverse, to the bound
\begin{equation}\label{EQ:Bound D0 2}
	\left\|			
	\begin{pmatrix}
	\sigma_a\\
	\sigma_b
	\end{pmatrix}
	-
	\begin{pmatrix}
	\wt\sigma_a \\
	\wt\sigma_b
	\end{pmatrix}
	\right\|_{L^2(\Omega)}
	\le 
	\left\|
	\begin{pmatrix}
	\Sigma_a^1 \\
	\Sigma_a^2 
	\end{pmatrix} 
	- 
	\begin{pmatrix}
	\wt \Sigma_a^1 \\
	\wt\Sigma_a^2 
	\end{pmatrix}
	\right\|_{L^2(\Omega)}
	+
	\|\bH-\wt \bH\|_{L^2(\Omega)}.
\end{equation}
The second bound of~\eqref{EQ:Stab sigma ab} then follows from~\eqref{EQ:Bound D0 1} and~\eqref{EQ:Bound D0 2}.
\end{proof}

%%%%%%%%%%%%%%%%%%%%%%%%%%%%%%%%%%%%%%%%%%%%%%%%%%%%%%%%%%%%%%%%%%
\subsection{Reconstruction with fixed-point iteration} 
%%%%%%%%%%%%%%%%%%%%%%%%%%%%%%%%%%%%%%%%%%%%%%%%%%%%%%%%%%%%%%%%%%

We now consider a fixed-point iteration algorithm for the reconstruction of the absorption coefficients. We again use the fact that if $(\sigma_a, \sigma_b, u)$ solves the transport equation~\eqref{EQ:ERT TP} to generate datum $H$, then we can replace the term $\sigma_a + \sigma_b\aver{u}$ in~\eqref{EQ:ERT TP} with $H / \aver{u}$ to obtain a nonlinear transport equation for $u$:
\begin{equation}\label{EQ:ERT IP}
\begin{array}{rcll}
	\bv \cdot \nabla u(\bx, \bv) +(\dfrac{H}{\aver{u}} +\sigma_s) u(\bx, \bv) &=& \sigma_s(\bx) Ku(\bx, \bv), &\mbox{in}\ X\, \\
u(\bx, \bv) &= & g(\bx,\bv), &\mbox{on}\ \Gamma_{-}\,.
\end{array}
\end{equation}
For a given datum $H$, if we could solve this equation, we can reconstruct $\sigma_a+\sigma_b\aver{u}$. Note that we have some \emph{a priori} bounds on $\aver{u}$ due to the \emph{a priori} bounds we know on the coefficients. First, it is clear that the coefficient $\Sigma_a$ to be reconstructed satisfies $\Sigma_a\ge \frac{H}{\overline g}$. Second, let us define
\[
	\eta(\bx):=\dfrac{H(\bx)}{\overline\sigma_a+\overline\sigma_b \overline g}\, .
\]
Then the transport solution $u$ that generated this datum $H$ satisfies: $\aver{u} \ge \eta$. Let $u_{\max}^H$ be the solution to the linear transport equation~\eqref{EQ:ERT} with $\Sigma_a=\frac{H}{\overline g}$. We then conclude, before we perform any reconstruction, that the solution~\eqref{EQ:ERT IP} that we are seeking has the property that
\[
	\eta\le \aver{u}\le \aver{u_{\max}^H}\, .
\]

Starting with a given $u_0$, we define the following iteration for $k\ge 1$:
\begin{equation}\label{EQ:ERT FP}
\begin{array}{rcll}
	\bv \cdot \nabla u_k(\bx, \bv) +(\dfrac{H}{\max(\aver{u_{k-1}}, \eta)} +\sigma_s) u_k(\bx, \bv) &=& \sigma_s(\bx) Ku_k(\bx, \bv), &\mbox{in}\ X\, \\
u_k(\bx, \bv) &= & g(\bx,\bv), &\mbox{on}\ \Gamma_{-}\,
\end{array}
\end{equation}
where the function $\max$ is applied point-wise to its arguments.

Let $u_{\min}$ be the solutions to the linear transport equation~\eqref{EQ:ERT} with absorption coefficient $\overline\sigma_a+\overline\sigma_b \overline g$. Here is an obvious observation on the iteration.
\begin{lemma}
	Let $\{u_k\}$ be a sequence generated by~\eqref{EQ:ERT FP} from an initial point $u_0\ge 0$. Then $u_{\min} \le u_k\le u_{\max}^H$, $\forall k\ge 1$.
\end{lemma}
\begin{proof}
We first observe that this iteration will generate a sequence $\{u_k\}$ such that $0\le u_k\le \overline g$, $\forall k\ge 1$. Therefore $\frac{H}{\overline g}\le \frac{H}{\max(\aver{u_{k-1}}, \eta)} \le \frac{H}{\eta}=\overline\sigma_a+\overline\sigma_b \overline g$. By monotonicity of the solution to the linear transport with respect to the absorption coefficient, we have $u_{\min}\le u_k\le u_{\max}^H$, $\forall k\ge 1$.
\end{proof}

We introduce the following space of functions with bounded angular average:
\begin{equation}
	\cU:=\{u \in L^2(X) \mid \aver{u_{\min}} \le \aver{u} \le \aver{u_{\max}^H} \ \mbox{a.e.}\}
\end{equation}
This space is convex, bounded and closed under the $L^2$ topology. We make the following assumption:\\[1ex]
$(\cA'')$ \hskip 6cm $\eta\le \aver{u_{\min}}$.\\[1.5ex]
We can then show the following result.
\begin{corollary}\label{COR:Extreme}
Let $\{\overline u_k\}$ and $\{\underline u_k\}$ be sequences generated from $\overline u_0=u_{\max}^H$ and $\underline u_0=u_{\min}$ respectively. Then $\overline u_k \to \overline u$ and $\underline u_k \to \underline u$ a.e. as $k\to \infty$ for some $\overline u, \underline u\in\cU$. 
\end{corollary}
\begin{proof}
For any sequence $\{u_k\}$ generated by ~\eqref{EQ:ERT FP}, let us define $\varphi_k:=u_k-u_{k-1}$. Then $\varphi_k$ solves
\begin{equation}\label{EQ:Iter Diff}
\begin{array}{rcll}
	\bv \cdot \nabla \varphi_k +(\dfrac{H}{\max(\aver{u_{k-1}}, \eta)} +\sigma_s) \varphi_k &=& \sigma_s(\bx) K \varphi_k+F, &\mbox{in}\ X\, \\
\varphi_k(\bx, \bv) &= & 0, &\mbox{on}\ \Gamma_{-}\,
\end{array}
\end{equation}
where 
\[
	F:=\dfrac{H u_{k-1} }{\max(\aver{u_{k-1}},\eta)\max(\aver{u_{k-2}},\eta)} [\max(\aver{u_{k-1}},\eta)-\max(\aver{u_{k-2}},\eta)]\,.
\]
When we start the iteration with $u_0\in\cU$, the iteration remains in $\cU$. Therefore $\aver{u_{k-1}}$, $\aver{u_{k-2}}\ge u_{\min}$. With the assumption $(\cA'')$, we conclude that 
\[
	\max(\aver{u_{k-1}},\eta)-\max(\aver{u_{k-2}},\eta)=\aver{u_{k-1}}-\aver{u_{k-2}}=\aver{\varphi_{k-1}}.
\]
Therefore, in this case $F:=\frac{H u_{k-1} }{\max(\aver{u_{k-1}},\eta)\max(\aver{u_{k-2}},\eta)}\aver{\varphi_{k-1}}$. For the iteration ~\eqref{EQ:ERT FP} that starts with $u_{\max}$, we have that $\varphi_0\le 0$. Therefore $\{\varphi_k\}$ remains negative according to~\eqref{EQ:Iter Diff}. This means that  $\{\overline u_k\}$ is a decreasing sequence. The fact that it is also bounded from below by $u_{\min}$ then indicates that it converges to some $\overline u\in\cU$. For the iteration ~\eqref{EQ:ERT FP} that starts with $u_{\min}$, we have that $\varphi_0\ge 0$. Therefore $\{\varphi_k\}$ remains non-negative according to~\eqref{EQ:Iter Diff}. This means that  $\{\underline u_k\}$ is an increasing sequence. The fact that it is also bounded from above by $u_{\max}^H$ then indicates that it converges to some $\underline u\in\cU$. 
\end{proof}

We are ready to show that the iteration~\eqref{EQ:ERT FP} converges to a unique fixed point in $\cU$ that is the solution to the transport equation~\eqref{EQ:ERT IP}.
\begin{theorem}
Assume that the solution to~\eqref{EQ:ERT IP} is such that $(\frac{H}{\aver{u}}, H, g)\in \Pi_\alpha'$ for some $\alpha$. Assume further that ${\Sigma_{a}}_{|\partial\Omega}$ is known and $\overline\Sigma_a \|\frac{g H}{{\Sigma_{a}}_{|\partial\Omega}}\|_{L^\infty(\Gamma_-)}\le \beta<1$ for some $\beta$. Then, under the assumption $(\cA)$, the iteration~\eqref{EQ:ERT FP} converges to the unique solution of~\eqref{EQ:ERT IP} in~$\cU$.
\end{theorem}
\begin{proof}
Let $\{\underline u_k\}$ be the sequence generated from the starting point $\underline u_0=u_{\min}$. By Corollary~\ref{COR:Extreme}, $\underline u_k\to \underline u$. Moreover $\underline u$ solves
\begin{equation*}
\begin{array}{rcll}
	\bv \cdot \nabla \underline u(\bx, \bv) +(\dfrac{H}{\max(\underline u, \eta)} +\sigma_s) \underline u  &=& \sigma_s(\bx) K \underline u(\bx, \bv), &\mbox{in}\ X\, \\
\underline u(\bx, \bv) &= & g(\bx, \bv), &\mbox{on}\ \Gamma_{-}\,.
\end{array}
\end{equation*}
We then use the fact that $\aver{\underline u}\in\cU$ and the assumption $(\cA'')$ to conclude that $\max(\underline u, \eta)=\underline u$. Therefore, $\underline u$ is a solution to ~\eqref{EQ:ERT IP}.

Let $\{u_k\}$ be a sequence generated from an arbitrary starting point in $\cU$. Let $\phi_k:=u_k-\underline u_k$. Then, using the same argument on $F$ in Corollary~\ref{COR:Extreme}, we check that $\phi_k$ solves
\begin{equation}
\begin{array}{rcll}
	\bv \cdot \nabla \phi_k +(\dfrac{H}{\max(\aver{u_{k-1}}, \eta)} +\sigma_s) \phi_k  &=& \sigma_s(\bx) K \phi_k+\wt F\aver{\phi_{k-1}}, &\mbox{in}\ X\, \\
\phi_k(\bx, \bv) &= & 0, &\mbox{on}\ \Gamma_{-}\,
\end{array}
\end{equation}
with $\wt F:=\frac{H\,\underline u_k}{\max(\aver{\underline u_{k-1}}, \eta)\max(\aver{u_{k-1}}, \eta)}$. This gives the bound:
\begin{multline}\label{EQ:Iter Bound}
	\| \phi_k \|_{L^2(X)} \le \| \wt F \aver{\phi_{k-1}}\|_{L^2(X)} \le \|\wt F\|_{L^\infty(X)} \|\phi_{k-1}\|_{L^2(X)} \\ 
	\le  \|\dfrac{H}{\max(\aver{u_{k-1}}, \eta)}\|_{L^\infty(X)} \|\dfrac{\underline u_{k-1}}{\max(\aver{\underline u_{k-1}}, \eta)}\|_{L^\infty(X)} \|\phi_{k-1}\|_{L^2(X)}.
\end{multline}
We first observe that $\|\frac{H}{\max(\aver{u_{k-1}}, \eta)}\|_{L^\infty(X)}\le \overline\Sigma_a$. To bound the term $\|\frac{\underline u_{k-1}}{\max(\aver{\underline u_{k-1}}, \eta)}\|_{L^\infty(X)}=\|\frac{\underline u_{k-1}}{\aver{\underline u_{k-1}}}\|_{L^\infty(X)}$, we observe that $w=\frac{\underline u_k}{\aver{\underline u_k}}$ solves the transport equation:
\begin{equation*}
\begin{array}{rcll}
	\bv \cdot \nabla w +(\dfrac{H}{\max(\aver{\underline u_{k-1}}, \eta)}+\bv\cdot\nabla \ln\aver{\underline u_k} +\sigma_s) w &=& \sigma_s(\bx) K w, &\mbox{in}\ X\, \\
w(\bx, \bv) &= & \dfrac{g H}{{\Sigma_a}_{|\partial\Omega}}, &\mbox{on}\ \Gamma_{-}\,
\end{array}
\end{equation*}
where the boundary condition for $w$ comes from the assumption that $\Sigma_{a}$ (and therefore the density $\aver{u_k}_{|\partial\Omega}=\dfrac{H}{{\Sigma_a}_{|\partial\Omega}}$) is known on the boundary of the domain. With the assumption that $(\frac{H}{\aver{\underline u}}, H, g)\in\Pi_\alpha'$, we have that $\frac{H}{\max(\aver{\underline u_{k-1}}, \eta)}+\bv\cdot\nabla \ln\aver{\underline u_k} \ge \frac{\alpha}{2}>0$ for sufficiently large $k$. Therefore, we conclude from the maximum principle that $\|w\|_{L^\infty(X)}\le \|\dfrac{gH}{{\Sigma_a}_{|\partial\Omega}}\|_{L^\infty(\Gamma_-)}$. Therefore the bound in~\eqref{EQ:Iter Bound} can now be written as
\[
	\| \phi_k \|_{L^2(X)} \le \overline\Sigma_a \|\dfrac{g H}{{\Sigma_a}_{|\partial\Omega}}\|_{L^\infty(\Gamma_-)}\|\phi_{k-1}\|_{L^2(X)}.
\]
When $\overline\Sigma_a \|\dfrac{g H}{{\Sigma_a}_{|\partial\Omega}}\|_{L^\infty(\Gamma_-)}\le \beta<1$, this bound gives that $\| \phi_k \|_{L^2(X)}\to 0$. This means that $u_k$ converges $\underline u$. 

The above calculation shows that the iteration~\eqref{EQ:ERT FP} sequence start with any initial point $u_0 \in \cU$ converges to a solution of~\eqref{EQ:ERT IP} whose average leaves in~$\cU$. This concludes the proof.
\end{proof}

The above result shows that in order to reconstruct the unknown absorption coefficients, we could use the fixed-point iteration~\eqref{EQ:ERT FP} to find $u$. We then reconstruct the total absorption $\sigma_a+\sigma_b\aver{u}$ from $H/\aver{u}$. This procedure would allow us to reconstruct $(\sigma_a, \sigma_b)$ from two different data sets $H_1$ and $H_2$. When we have a better \emph{a priori} information on the coefficient to be reconstructed, we could modify $\eta$ to further reduce the size of the space $\cU$. This will in turn allow us to better reconstruct $u$.

When the media scatters isotropically, that is, when $\Theta(\bv, \bv')\equiv 1$, we could make some of the assumptions we made in this section more explicit. The calculations are documented in the Appendix~\ref{SEC:Appendix B}.

%%%%%%%%%%%%%%%%%%%%%%%%%%%%%%%%%%%%%%%%%%%%%%%%%%%%%%%%%%%%%%%%%%
%%%%%%%%%%%%%%%%%%%%%%%%%%%%%%%%%%%%%%%%%%%%%%%%%%%%%%%%%%%%%%%%%%
\section{Concluding remarks}
\label{SEC:Concl}
%%%%%%%%%%%%%%%%%%%%%%%%%%%%%%%%%%%%%%%%%%%%%%%%%%%%%%%%%%%%%%%%%%
%%%%%%%%%%%%%%%%%%%%%%%%%%%%%%%%%%%%%%%%%%%%%%%%%%%%%%%%%%%%%%%%%%

In this work, we analyzed an inverse problem for a semilinear radiative transport equation, aiming at reconstructing two absorption coefficients of the transport equation from two internal data sets that are functionals of the transport solutions. We first established the well-posedness of the forward problem under small boundary sources. We then derived stability results on the inverse problem in the simplified settings where the scattering coefficient is known (either $\sigma_s\equiv 0$ or $\sigma_s>0$). We also developed a reconstruction method based on a fixed-point iteration. Our results provide some mathematical understanding of quantitative photoacoustic imaging of two-photon absorption in the transport regime, complementing the results in~\cite{ReZh-SIAM18} in the diffusive regime.

There are several interesting following up questions to the current work. For instance, it would be useful if we can remove some of the restrictive assumptions on the size of the gradient of the absorption coefficients to be reconstructed. Moreover, it would be of great interests to generalize the analysis we have to reconstruct simultaneously the absorption and the scattering coefficients triplet $(\sigma_a, \sigma_b, \sigma_s)$ from three sets of internal data. Note that in the case of linear transport equation, i.e.~\eqref{EQ:ERT TP} without the semilinear term, the analysis in~\cite{BaJoJu-IP10} shows that one can reconstruct $(\sigma_a, \sigma_s)$ as well as partial information in the scattering phase function $\Theta(\bv, \bv')$ with data encoded in the full operator $\Lambda: g(\bx, \bv) \mapsto H(\bx)$. Whether or not one can reconstruct simultaneously $\sigma_a$ and $\sigma_s$ in the linear transport equation from a finite number of internal data is still a largely open question right now; see some progresses in~\cite{MaRe-CMS14, HaNeNgRa-SIAM18}. From application point of view, it is an interesting problem to see if one can reconstruct all the coefficients in the problem from the albedo data $\Lambda: u_{|\Gamma_-} \mapsto u_{\Gamma_+}$. This can probably be analyzed by combining the classical singular decomposition of Choulli and Stefanov~\cite{ChSt-CPDE96,ChSt-IP96} with the linearization idea introduced by Isakov and collaborators~\cite{IsNa-TAMS95,IsSy-CPAM94,Sun-MZ96}.

%%%%%%%%%%%%%%%%%%%%%%%%%%%%%%%%%%%%%%%%%%%%%%%%%%%%%%%%%%%%%%%%%%
%%%%%%%%%%%%%%%%%%%%%%%%%%%%%%%%%%%%%%%%%%%%%%%%%%%%%%%%%%%%%%%%%%
\section*{Acknowledgments}
%%%%%%%%%%%%%%%%%%%%%%%%%%%%%%%%%%%%%%%%%%%%%%%%%%%%%%%%%%%%%%%%%%
%%%%%%%%%%%%%%%%%%%%%%%%%%%%%%%%%%%%%%%%%%%%%%%%%%%%%%%%%%%%%%%%%%

This work is partially supported by the National Science Foundation through grants DMS-1913309 and DMS-1937254. %We would like to thank the anonymous referees for their useful comments that help us improve the quality of this work.

\appendix
\renewcommand{\thesection}{\Alph{section}}
%\setcounter{section}{0}
%%%%%%%%%%%%%%%%%%%%%%%%%%%%%%%%%%%%%%%%%%%%%%%%%%%%%%%%%%%%%%%%%%
%%%%%%%%%%%%%%%%%%%%%%%%%%%%%%%%%%%%%%%%%%%%%%%%%%%%%%%%%%%%%%%%%%
\section{Averaging lemma and Kellogg's theory}
\label{SEC:Appendix A}
%%%%%%%%%%%%%%%%%%%%%%%%%%%%%%%%%%%%%%%%%%%%%%%%%%%%%%%%%%%%%%%%%%
%%%%%%%%%%%%%%%%%%%%%%%%%%%%%%%%%%%%%%%%%%%%%%%%%%%%%%%%%%%%%%%%%%

To improve the readability of the paper, we recall here two important results that we have used to prove the main results of the paper. 

The first result is the averaging lemma in transport theory, developed in~\cite{GoLiPeSe-JFA88}. It characterizes the regularization effect of velocity averaging on the solution of transport equations. With the same notations as in the main text, the result can be stated as follows.
\begin{theorem}[Averaging Lemma]\label{THM:Averaging Lemma}
    For $p\in(1, +\infty)$, let $u$ be a function defined in $X$ such that $u\in L^p(X)$, $\bv\cdot \nabla u \in L^p(X)$, and $u|_{\Gamma_{-}}\in L^{p}(\Gamma_{-})$. Then $\aver{u}$ belongs to the Sobolev space $W^{s,p}(\Omega)$ with $s=1/2$ if $p=2$ and $0 < s < \inf(p^{-1}, 1-p^{-1})$ if $p\neq 2$. In addition, we have the inequality  
    \begin{equation}
    \|\aver{u}\|_{W^{s,p}(\Omega)} \le \fc \left( \|u\|_{L^p(X)} + \|\bv\cdot \nabla u\|_{L^p(X)} + \|u\|_{L^{p}(\Gamma_{-})}\right)
\,,
    \end{equation}
    for some constant $\fc>0$.
\end{theorem}

The second result we recall here is Kellogg's uniqueness theory for the Schauder Fixed-Point Theorem, developed in~\cite{Kellogg-PAMS76}. The theory provides a condition under which the Schauder fixed point is unique.
\begin{theorem}[Kellogg 1976~\cite{Kellogg-PAMS76}]\label{THM:UNIQ}
	Let $\cM$ be a bounded convex open subset of a real Banach space, and $F:\overline{\cM}\to\overline{\cM}$ a compact continuous map which is continuously Fr\'echet differentiable on $\cM$. If (i) for each $m\in \cM$, $1$ is not an eigenvalue of $F'(m)$, and (ii) for each $m\in\partial\cM$, $m \neq F(m)$, then $F$ has a unique fixed point in $\cM$. 
\end{theorem}

%%%%%%%%%%%%%%%%%%%%%%%%%%%%%%%%%%%%%%%%%%%%%%%%%%%%%%%%%%%%%%%%%%
\section{Inversion in isotropic media} 
\label{SEC:Appendix B}
%%%%%%%%%%%%%%%%%%%%%%%%%%%%%%%%%%%%%%%%%%%%%%%%%%%%%%%%%%%%%%%%%%

We analyze here the inverse problem in Section~\ref{SEC:IP Scattering} in the context of isotropic scattering. This is again done by analyzing a fixed-point iteration for solving~\eqref{EQ:ERT IP}. For simplicity, in the following we consider the boundary source $g(\bx, \bv)\equiv \overline{g}$ as a constant.

For any positive function $m(\bx)>0$, we define a map $\cC$ through the relation:
\[
	\cC(m):=\aver{u},
\]
where $u$ solves the following linear transport equation with isotropic scattering:
\begin{equation}\label{EQ:ERT FP 2}
\begin{array}{rcll}
\bv \cdot \nabla u(\bx, \bv) +({H}/{m} + \sigma_s) u(\bx, \bv) &=& \sigma_s(\bx)  \dint_{\bbS^{d-1}} u(\bx, \bv) d\bv, &\mbox{in}\ X\,\\
u(\bx, \bv) &= & \overline{g}, &\mbox{on}\  \Gamma_{-}\,.
\end{array}
\end{equation}
The map $\cC$ is monotone increasing and is bounded from above by $\overline{g}$ in the space of positive functions, under the assumptions in $(\cA)$. We will show that $\cC$ admits a unique fixed point in appropriate sense.

Let $\eta$ be defined as in Section~\ref{SEC:IP Scattering} and satisfy the assumption $(\cA'')$, that is, $\cC(\eta) \ge \eta$.
We define the function space
\begin{equation}\label{EQ:M Inv}
	\cM=\{ m\in L^2(\Omega) \mid \eta \le m \le \overline{g}\quad a.e.\}\,.
\end{equation}
Then $\cC$ is monotone increasing and $\cC(\cM)\subset \cM$. Moreover,  $\cC$ is compact and continuous on $\cM$ in $L^{2}(\Omega)$ topology. The existence of solution on $\cC$ then follows from the Schauder Fixed-Point Theorem. 

To show the uniqueness of the fixed point of $\cC$, we first observe that the equation~\eqref{EQ:ERT FP 2} is equivalent to the following integral equation:
\begin{equation}
\aver{u}(\bx) = \cJ_m \overline{g} + \cK_{m} (\sigma_s\aver{u}),
\end{equation}
where the integral operators $\cJ_m:L^p(\Gamma_{-})\to L^p(\Omega)$ and $\cK_m:L^p(\Omega)\to L^p(\Omega)$, $1\le p\le \infty$, are defined as follows:
\begin{equation}
\begin{aligned}
\cJ_m \overline{g} &= \overline{g}\int_{\bbS^{d-1}}  E_{m}(\bx, \tau_{-}(\bx, \bv), \bv) d\bv\,,\\
\cK_m f &= \int_{\bbS^{d-1}} \int_{0}^{\tau_{-}(\bx, \bv)} E_{m}(\bx, l, \bv) f(\bx - l\bv) dl d\bv\,,
\end{aligned}
\end{equation}
with the path integral operator $E_m$ given as 
$$E_m(\bx, l, \bv) = \exp\left[-\displaystyle\int_0^{l} (\frac{H}{m}+\sigma_s)(\bx - s\bv) ds\right].$$

We first show the result on the fixed point starting from $\overline{g}$.
\begin{lemma}\label{LEM: SCATTER H}
    Let $h := \dlim_{n\to\infty} \cC^n (\overline{g})$. For any $\psi\in L^{\infty}(\Omega)$ such that $\psi \ge 0$, the corresponding integral operators $\cJ_h$ and $\cK_h$ satisfy
    \begin{equation}
    J_h\overline{g}\ge \frac{h-\mu_h}{\overline{g}-\mu_h} \overline{g} ,\quad \cK_h \psi \le \frac{\overline{g} - h}{\overline{g} - \mu_h}\sup\left[\frac{\psi}{\frac{H}{h}+\sigma_s}\right] \mbox{ with } \mu_h = \sup\left[\frac{\sigma_s h}{\frac{H}{h}+\sigma_s}\right]\,.
    \end{equation} 
\end{lemma}
\begin{proof}
    We first observe that:
    \begin{multline}\label{EQ: KH}
    \cK_h \psi = \int_{\bbS^{d-1}} \int_{0}^{\tau_{-}(\bx, \bv)} E_h(\bx, l, \bv) \psi(\bx - l\bv) dl d\bv  \\
    \le \int_{\bbS^{d-1}} \left(1 - E_h(\bx, \tau_{-}(\bx, \bv), \bv)\right)\sup \left[ \frac{\psi}{(\frac{H}{h}+\sigma_s)} \right] d\bv%\\
    \le \left(1 - \frac{1}{\overline{g}}\cJ_{h}\overline{g}\right) \sup \left[ \frac{\psi}{(\frac{H}{h}+\sigma_s)} \right].
    \end{multline}
    The function $h$ solves the integral equation $h = \cJ_h \overline{g} + \cK_h(\sigma_s h)$. Therefore,
    \begin{equation}
    \begin{aligned}
    h \le \cJ_h\overline{g} + \left(1 - \frac{1}{\overline{g}}\cJ_{h}\overline{g}\right) \sup \left[ \frac{\sigma_s h}{(\frac{H}{h}+\sigma_s)} \right],
    \end{aligned}
    \end{equation}
    which means
    \begin{equation}
    \left(1 - \frac{1}{\overline{g}} \sup \left[ \frac{\sigma_s h}{(\frac{H}{h}+\sigma_s)} \right] \right)\cJ_h \overline{g} \ge h - \sup \left[ \frac{\sigma_s h}{(\frac{H}{h}+\sigma_s)} \right].
    \end{equation}
    The proof is completed by bringing the above inequality into~\eqref{EQ: KH}.
\end{proof}

\begin{lemma}\label{LEM: SCATT EST}
	Assume that $\ell_{\Omega} = \mbox{diam}(\Omega)\le 1$. Let $h := \dlim_{n\to\infty}\cC^n(\overline{g})$ and $f$ be an arbitrary element of $\cM$. We have that
    \begin{equation}
    |\cC(h) - \cC(f)| \le \gamma \left[  \kappa\mu_f[1 - (1-\ell_{\Omega})\frac{h-\mu_h}{\overline{g}-\mu_h}] + \frac{\overline{g}-h}{\overline{g}-\mu_h}\overline{g}\right],
    \end{equation}
    where $\gamma = \sup\left|\frac{h-f}{h\wedge f}\right|$, $\kappa = \sup\frac{\frac{H}{h\vee f}}{\frac{H}{h\vee f}+\sigma_s}$, $\mu_f = \sup\frac{\sigma_s f}{\frac{H}{h\vee f} + \sigma_s}$, and $\mu_h = \sup\frac{\sigma_s h}{\frac{H}{h}+\sigma_s}$ with the notations $h\wedge f := \min(h, f)$ and $h\vee f := \max(h,f)$. %Note the definition of $\mu_f$ is not the same way as $\mu_h$ unless $h\ge f$.
\end{lemma}
\begin{proof}
    We need to bound $|\cC(h) - \cC(f)|$. We first observe that
    \begin{equation*}
    \begin{aligned} 
    |\cC(h) - \cC(f)| &\le |\cJ_h \overline{g} - \cJ_f \overline{g}| + |\cK_h(\sigma_s(h-f))| + |(\cK_h - \cK_f)(\sigma_s f)|\\
    &\equiv A_1 + A_2 + A_3\,.
    \end{aligned}
    \end{equation*}
   Using the fact that $|e^{-x} - e^{-y}|\le e^{-\min(x,y)}|x-y|$, $\forall x, y\in \bbR$, we have
   \begin{equation*}
   | E_h(\bx, l, \bv) - E_f(\bx, l, \bv) |\le E_{h\vee f}(\bx, l, \bv)  \left|\int_{0}^l \frac{H(h-f)}{hf}(\bx - s\bv) ds\right|.
   \end{equation*}
   Note that $hf = (h\vee f)(h\wedge f)$, we obtain the estimates for $A_1$ and $A_2$:
    \begin{equation*}
    \begin{aligned}
    A_1 &= |\cJ_h \overline{g} - \cJ_f \overline{g}|\\
    &\le \overline{g}\int_{\bbS^{d-1}} E_{h\vee f}(\bx, \tau_{-}(\bx, \bv), \bv)\left|\int_{0}^{\tau_{-}(\bx, \bv)}\frac{H(h-f)}{hf}(\bx - s\bv)ds\right| d\bv \\
    &\le \overline{g}\int_{\bbS^{d-1}} \left(\int_{0}^{\tau_{-}(\bx, \bv)}E_{h\vee f}(\bx, s, \bv)\left|\frac{H(h-f)}{hf}(\bx - s\bv)\right|ds\right) d\bv\\
    &\le \overline{g} \sup\left|\frac{h-f}{h\wedge f}\right| \int_{\bbS^{d-1}} \left(\int_{0}^{\tau_{-}(\bx, \bv)}E_{h\vee f}(\bx, s, \bv)\left|\frac{H}{h\vee f}(\bx - s\bv)\right|ds\right) d\bv\,,\\
    A_2 &= |\cK_h(\sigma_s(h - f) )| \\
    &\le\cK_h(\sigma_s (h\wedge f) \sup\left|\frac{h-f}{h\wedge f}\right|)\\
    &\le \cK_{h\vee f}(\sigma_s (h\wedge f) \sup\left|\frac{h-f}{h\wedge f}\right|)\\
    &\le  \sup\left|\frac{h-f}{h\wedge f}\right|\int_{\bbS^{d-1}} \left(\int_{0}^{\tau_{-}(\bx, \bv)}E_{h\vee f}(\bx, s, \bv)\left|\sigma_s (h\wedge f)(\bx - s\bv)\right|ds\right) d\bv\,,
    \end{aligned} 
    \end{equation*}
     The above estimates imply that
    \begin{equation*}
    \begin{aligned}
    A_1 + A_2 &\le  \sup\left|\frac{h-f}{h\wedge f}\right|\int_{\bbS^{d-1}} \left(\int_{0}^{\tau_{-}(\bx, \bv)}E_{h\vee f}(\bx, s, \bv)\left|\frac{\overline{g}H}{h\vee f}+\sigma_s (h\wedge f)\right|(\bx - s\bv)ds\right) d\bv\\
    &\le  \overline{g} \sup\left|\frac{h-f}{h\wedge f}\right|\left(1 - E_{h\vee f}(\bx ,\tau_{-}(\bx, \bv), \bv)\right)\sup\frac{\frac{H}{h\vee f}+\sigma_s \frac{h\wedge f}{\overline{g}}}{\frac{H}{h\vee f} + \sigma_s} \\
    &\le \overline{g} \sup\left|\frac{h-f}{h\wedge f}\right|\left(1 - E_{h\vee f}(\bx ,\tau_{-}(\bx, \bv), \bv)\right)\,.
    \end{aligned}
    \end{equation*}
    %Let $\gamma = \sup\left|\frac{h-f}{h\wedge f}\right|$, $\kappa = \sup\frac{\frac{H}{h\vee f}}{\frac{H}{h\vee f}+\sigma_s}$, $\mu_f = \sup\frac{\sigma_s f}{\frac{H}{h\vee f} + \sigma_s}$, $\ell_{\Omega} = \mbox{diam}(\Omega)\le 1$, 
    To estimate $A_3$, we observe that:
    \begin{equation*}
    %{\small
    \begin{aligned}
    A_3 &\le (\cK_h - \cK_f)(\sigma_s f) \\
    &\le \int_{\bbS^{d-1}} \int_0^{\tau_{-}(\bx, \bv)} E_{h\vee f}(\bx, l, \bv) \left|\int_0^l \frac{H(h-f)}{hf}(\bx - s\bv) ds \right|\sigma_s f(\bx - l\bv) dl d\bv \\
    &\le \gamma \int_{\bbS^{d-1}} \int_0^{\tau_{-}(\bx, \bv)} E_{h\vee f}(\bx, l, \bv) \left(\int_0^l \frac{H}{h\vee f}(\bx - s\bv) ds \right)\sigma_s f(\bx - l\bv) dl d\bv\\
    &\le \gamma \kappa \int_{\bbS^{d-1}} \int_0^{\tau_{-}(\bx, \bv)} E_{h\vee f}(\bx, l, \bv) \left(\int_0^l (\frac{H}{h\vee f}+\sigma_s)(\bx - s\bv) ds \right)\sigma_s f(\bx - l\bv) dl d\bv \\
    &\le \gamma \kappa \mu_f\int_{\bbS^{d-1}} \int_0^{\tau_{-}(\bx, \bv)} E_{h\vee f}(\bx, l, \bv) \left(\int_0^l (\frac{H}{h\vee f}+\sigma_s)(\bx - s\bv) ds \right)(\frac{H}{h\vee f}+\sigma_s)(\bx - l\bv)dl d\bv \\
    &\le \gamma \kappa \mu_f \left[1 - (1-\ell_{\Omega}) \int_{\bbS^{d-1}} E_{h\vee f}(\bx, \tau_{-}(\bx, \bv), \bv) d\bv   \right]\,.
    \end{aligned}
    %}
    \end{equation*}
    We can then use the fact that $E_{h\vee f}(\bx, l, \bv)\ge E_h(\bx, l,\bv)$ and Lemma~\ref{LEM: SCATTER H} to obtain that,
    \begin{equation*}
    \begin{aligned}
    & |\cC(h) - \cC(f)| \\
    &\le \gamma \left[ \overline{g} - \overline{g}\int_{\bbS^{d-1}} E_{h\vee f}(\bx, \tau_{-}(\bx, \bv), \bv) d\bv + \kappa \mu_f[1 - (1-\ell_{\Omega})\int_{\bbS^{d-1}} E_{h\vee f} (\bx, \tau_{-}(\bx, \bv), \bv)d\bv ] \right]\\
    &\le \gamma \left[ \overline{g} - \overline{g}\int_{\bbS^{d-1}} E_{h}(\bx, \tau_{-}(\bx, \bv), \bv) d\bv + \kappa \mu_f[1 - (1-\ell_{\Omega})\int_{\bbS^{d-1}} E_{h} (\bx, \tau_{-}(\bx, \bv), \bv)d\bv ] \right] \\
    &\le \gamma \left[  \kappa\mu_f[1 - (1-\ell_{\Omega})\frac{h-\mu_h}{\overline{g}-\mu_h}] + \frac{\overline{g}-h}{\overline{g}-\mu_h}\overline{g}\right]\,.
    \end{aligned}
    \end{equation*}
    This completes the proof.
\end{proof}

We are now ready to establish the uniqueness result. 
We define
\[
	\alpha := \sup\frac{{H}/{\eta}}{{H}/{\eta}+\sigma_s} \quad\mbox{and}\quad \beta := \sup\frac{\sigma_s}{{H}/{\overline{g}}+\sigma_s}\,.
\]
\begin{theorem}\label{THM: SCAT UNIQ}
   Assume that $\ell_{\Omega} := \mbox{diam}(\Omega) \le 1$. Let $\psi$ be defined as
   \begin{equation}
   	\psi = \frac{1 + \alpha \beta - \alpha\ell_{\Omega}\beta^2}{2 - [1-(1-\ell_{\Omega})\alpha  ]\beta}\overline{g}. 
   \end{equation}
	When $\psi\le \eta$, the transport equation~\eqref{EQ:ERT IP} admits a unique solution.
\end{theorem}
\begin{proof}
    Let $h := \dlim_{n\to\infty} \cC^n (\overline{g})$ and $f := \dlim_{n\to\infty} \cC^n (\eta)$. It is clear that $\overline{g}\ge h\ge f\ge \eta$, $h\vee f = f$ and $h\wedge f = f$. Let $\gamma$, $\kappa$, $\mu_f$ and $\mu_h$ be given as in Lemma~\ref{LEM: SCATT EST}. Then we have
    \begin{equation}\label{EQ: ALL EST}
    \begin{aligned}
    \frac{h-f}{f}\le \gamma\left( \frac{ \overline{g}}{f}\frac{\overline{g} - h}{\overline{g} - \mu_h} + \frac{1}{f}\kappa  \mu_f [1 - (1-\ell_{\Omega})\frac{h - \mu_h}{\overline{g}-\mu_h}]\right).
    \end{aligned}
    \end{equation}
    To obtain the uniqueness, it is sufficient to have
    \begin{equation}\label{EQ: ALL EST2}
    \frac{ \overline{g}}{f}\frac{\overline{g} - h}{\overline{g} - \mu_h} + \frac{1}{f}\kappa  \mu_f [1 - (1-\ell_{\Omega})\frac{h - \mu_h}{\overline{g}-\mu_h}] \le  1\,.
    \end{equation}
    The equal sign case is safely included. This is because if the equal sign holds at some point $\bz\in\Omega$, then we need, from the estimate in $A_1$, that
    \begin{equation}
    \int_0^{\tau_{-}(\bz, \bv)}\frac{H(h-f)}{hf}(\bz - s\bv)ds = 0\,,
    \end{equation}
for all $\bv\in\bbS^{d-1}$. This means that $h - f\equiv 0$. 

The inequality~\eqref{EQ: ALL EST2} is equivalent to:
    \begin{equation}\label{EQ: KEY}
    t \kappa\ell_{\Omega}\mu_h + (1-t)\overline{g} \le tf + (1-t)h\,,
    \end{equation}
	with $t =  \frac{(\overline{g}-\mu_h)}{2\overline{g} - \mu_h + \kappa\mu_f(1-\ell_{\Omega})}\in (0,1)$. Using the definitions of the parameters, we have that $\mu_f\le \mu_h \le \beta\overline{g}$. This allows us to conclude that the left-hand-side of~\eqref{EQ: KEY} is bounded by 
\begin{equation}
     t \kappa\ell_{\Omega}\mu_h + (1-t)\overline{g}  \le\frac{1 + \kappa \beta - \kappa \ell_{\Omega}\beta^2}{2 - [1-(1-\ell_{\Omega})\kappa ]\beta}\overline{g}\le  \psi\,,
\end{equation}
since $\alpha = \sup\frac{{H}/{\eta}}{{H}/{\eta}+\sigma_s}\ge \kappa$. Now we can use the assumptions $\psi\le \eta$ and the facts that $\eta\le f\le h$ to deduce~\eqref{EQ: KEY}.
\end{proof}
%\begin{remark}\rm
  %  When $\sigma_s\to 0$, the above result degenerates to that in Remark~\ref{REM: NONSCAT UNIQ}. 
%\end{remark}

The next theorem gives the stability of the solution $u$ with respect to changes in $H$. 
\begin{theorem}\label{THM:STAB}
Let $\wt u$ and $\wh{u}$ be the unique solutions to ~\eqref{EQ:ERT IP} with internal data $\wt H$ and $\wh{H}$ respectively but the same source function $g \equiv \overline{g}$. Assume that $\ell_{\Omega} := \mbox{diam}(\Omega) \le 1$ and there exists a constant $0\le r < 1$, such that
\begin{equation*}
\psi := \frac{1 + \alpha \beta - \alpha\ell_{\Omega}\beta^2}{1 + r(1-\beta)  + (1-\ell_{\Omega})\alpha\beta}\overline{g} \le \aver{u}, \aver{\hat u}.
\end{equation*}
Then there exists a constant $c > 0$ that
\begin{equation}\label{EQ:Sol Stab Bound}
\left\| \aver{\wt u} - \aver{\wh{u}}  \right\|_{L^{\infty}(\Omega)} \le c \|{\wt H - \wh{H}}\|_{L^{\infty}(\Omega)}.
\end{equation}
\end{theorem}
\begin{proof}
For a given function $m(\bx)>0$, let $\wt w$ and $\wh{w}$ be respectively the solution to the transport equations
\begin{equation}\label{EQ: WJ 1}
\begin{array}{rcll}
\bv \cdot \nabla \wt w(\bx, \bv) +(\wt H/m +\sigma_s) \wt w(\bx, \bv) &=& \sigma_s(\bx) \dint_{\bbS^{d-1}} \wt w(\bx, \bv) d\bv, &\mbox{in}\ X\, \\
\wt w(\bx, \bv) &= & \overline{g} &\mbox{on}\  \Gamma_{-}\,
\end{array}
\end{equation}
and
\begin{equation}\label{EQ: WJ 2}
\begin{array}{rcll}
\bv \cdot \nabla \wh{w}(\bx, \bv) +(\wh{H}/{m} +\sigma_s) \wh{w}(\bx, \bv) &=& \sigma_s(\bx) \dint_{\bbS^{d-1}} \wh w(\bx, \bv) d\bv, &\mbox{in}\ X\, \\
\wh{w}(\bx, \bv) &= & \overline{g}  &\mbox{on}\ \Gamma_{-}\,.
\end{array}
\end{equation}
We define the maps $\widetilde{\cC}$ and $\widehat{\cC}$ via the relations:
\begin{equation*}
\widetilde{\cC}(m) := \aver{\wt w}\quad\mbox{and}\quad \widehat{\cC}(m) := \aver{\wh{w}}.
\end{equation*}
Then $\aver{\wt{u}} =\dlim_{n\to\infty} \widetilde{\cC}^n(\overline{g})$ and  $ \aver{\wh{u}} =\dlim_{n\to\infty} \widehat{\cC}^n(\overline{g})$. For convenience, we denote $h := \aver{\wt{u}}$ and $f:=\aver{\wh{u}}$. It is then straightforward to check that
\begin{equation}\label{EQ:Stab Bound}
\begin{aligned}
\left|\frac{h-f}{h\wedge f}\right| &= \left|\frac{\widetilde{\cC}(h) - \widehat{\cC}(f)}{h\wedge f}\right|\le  \left|\frac{\widetilde{\cC}(h) - \widetilde{\cC}(f)}{h\wedge f}\right| +\left| \frac{\widetilde{\cC}(f) - \widehat{\cC}(f)}{h\wedge f}\right|\,.
\end{aligned}
\end{equation}
We first bound the first part on the right hand side. Following Lemma~\ref{LEM: SCATT EST}, we define $\gamma = \sup\left|\frac{h-f}{h\wedge f}\right|$, $\kappa = \sup\frac{{\wt{H}_j}/{(h\vee f)}}{{\wt{H}_j}/{(h\vee f)} +\sigma_s}$, $\mu_h = \sup\frac{\sigma_s h}{{\wt{H}_j}/{h} +\sigma_s}$, and $\mu_f = \sup\frac{\sigma_s h}{{\wt{H}_j}/{(h\vee f)} +\sigma_s}$. Then $0\le \mu_h, \mu_f \le \beta \overline{g}_j$ and $\alpha \ge \kappa$. 
We obtain
\begin{equation*}
\begin{aligned}
\left|\frac{\widetilde{\cC}(h) - \widetilde{\cC}(f)}{ h\wedge f }\right| &\le \frac{\gamma }{h\wedge f }\left[  \kappa\mu_f[1 - (1-\ell_{\Omega})\frac{h-\mu_h}{\overline{g}_j-\mu_h}] + \frac{\overline{g}_j-h}{\overline{g}_j-\mu_h}\overline{g}_j\right]\\
&\le \frac{\gamma}{h\wedge f} \left[\alpha \beta[\overline{g}_j- (1-\ell_{\Omega})\frac{h - \beta\overline{g}_j}{(1-\beta)}] + \frac{\overline{g}_j - h}{1- \beta} \right]\,.
\end{aligned}
\end{equation*}
When $h, f\ge \psi$, it is easy to check that
\begin{equation*}
\frac{1}{h\wedge f} \left[\alpha \beta[\overline{g}_j- (1-\ell_{\Omega})\frac{h - \beta\overline{g}_j}{(1-\beta)}] + \frac{\overline{g}_j - h}{1- \beta} \right] \le r < 1\,.
\end{equation*}
Hence
\begin{equation}\label{EQ:Stab Bound2}
\left|\frac{\widetilde{\cC}(h) - \widetilde{\cC}(f)}{ h\wedge f }\right| \le r\sup\left|\frac{h-f}{h\wedge f}\right|.
\end{equation}

To bound the second term on the right hand side of~\eqref{EQ:Stab Bound}, we take $m = f$ in both~\eqref{EQ: WJ 1} and~\eqref{EQ: WJ 2} and take the difference $\phi = \wt{w} - \wh{w}$. Simple algebra shows that $\phi$ satisfies the following equation:
\begin{equation*}
\begin{array}{rcll}
\bv \cdot \nabla \phi(\bx, \bv) +(\wt{H}/{f} +\sigma_s) \phi(\bx, \bv) &=& \sigma_s(\bx)\dint_{\bbS^{d-1}} \phi(\bx,\bv)d\bv - \frac{\wt{H} - \wh{H}}{f}\wh{w}, &\mbox{in}\ X\, \\
\phi(\bx, \bv) &= & 0, &\mbox{on}\ \Gamma_{-}\,.
\end{array}
\end{equation*}
By standard results from transport theory~\cite{DaLi-Book93-6}, there is a constant $c > 0$ that
\begin{equation*}
\|\phi\|_{L^{\infty}(\Omega)} \le c\|{\wt{H} - \wh{H}}\|_{L^{\infty}(\Omega)}.
\end{equation*}
Therefore, we have
\begin{equation*}
\begin{aligned}
	\left|\frac{\widetilde{\cC}(f) - \widehat{\cC}(f)}{h\wedge f }\right|&\le  \frac{c}{\psi}\|{\wt{H} - \wh{H}}\|_{L^{\infty}(\Omega)}.
\end{aligned}
\end{equation*}
Combining this with~\eqref{EQ:Stab Bound2}, we have the following stability bound:
\begin{equation*}
\left\|{h-f}\right\|_{L^{\infty}(\Omega)} \le \frac{c\overline{g}}{(1-r)\psi}\|{\wt{H}_j - \wh{H}_j}\|_{L^{\infty}(\Omega)}\le \frac{2c}{1-r}\|{\wt{H}_j - \wh{H}_j}\|_{L^{\infty}(\Omega)}.
\end{equation*}
The proof is complete.
\end{proof}

{\small
\bibliography{C:/RenGDrive/Publications/Bibliography/RH-BIB}
\bibliographystyle{siam}
}
%%%%%%%%%%%%%%%%%%%%%%%%%%%%%%%%%%%%%%%%%%%%%%%%%%%%%%%%%%%%%%%%%%
%%%%%%%%%%%%%%%%%%%%%%%%%%%%%%%%%%%%%%%%%%%%%%%%%%%%%%%%%%%%%%%%%%

\end{document}